\documentclass[11pt, twoside]{article}

\usepackage{psfrag}
\usepackage{color}
\usepackage{amsmath, amsfonts, amssymb, amsthm}
\usepackage{graphicx}
\usepackage{rotating}
\usepackage{multirow}
\usepackage{relsize}
\usepackage{subcaption}

\usepackage{booktabs}
\usepackage{hyperref} 
\usepackage{arydshln}
\usepackage{multicol}
\usepackage{overpic}
\usepackage{stmaryrd}
\usepackage{pict2e}
\usepackage{pdfpages}
\usepackage{pgfplots}
\usepackage{pgfplotstable}
    \usepgfplotslibrary{groupplots}
    \usetikzlibrary{spy,backgrounds}
    \pgfplotsset{compat=1.18}
    
\usepackage{todonotes}
    \presetkeys{todonotes}{inline}{}

\usepackage{isomath}

\DeclareMathOperator{\dive}{div}
\DeclareMathOperator{\grad}{\nabla}

\newcommand{\sumN}{\sum_{i=0}^N}
\newcommand{\sumNoverlap}{\sum_{i=0}^N}
\newcommand{\sumEij}{\sum_{E_{ij}}}

\newcommand{\jump}[1]{\llbracket #1 \rrbracket}
\newcommand{\avgdomainu}[1]{\bar u_{\Omega_{#1}}}

\newcommand{\Eij}{E_{ij}}
\newcommand{\edgeij}{\mathcal E^{ij} }
\newcommand{\baseedge}[1]{\vartheta_{E_{#1}}}
\newcommand{\avgedgeu}[1]{\bar u_{E_{#1}}}

\newcommand{\vertexl}{\mathcal V^l }
\newcommand{\oververtexl}{\overline{\mathcal V}^l}
\newcommand{\undervertexl}{\underline{\mathcal{V}}^l}
\newcommand{\basevertexl}{\vartheta^l}
\newcommand{\overbasevertexl}{\bar{\vartheta}^l}
\newcommand{\underbasevertexl}{\underline{\vartheta}^l}

\newcommand{\R}{\mathbb R}
\newcommand{\Hnorm}[2]{\| #1 \|_{H^1(\Omega_{#2})}}
\newcommand{\Hseminorm}[2]{| #1 |_{H^1(\Omega_{#2})}}

\newtheorem{remark}{Remark}
\newtheorem{theorem}{Theorem}
\newtheorem{lemma}{Lemma}

\title{GDSW preconditioners for composite Discontinuous Galerkin discretizations of multicompartment 
reaction-diffusion problems}

\author{
	Ngoc Mai Monica Huynh\thanks{
	    Dipartimento di Matematica, Universit\`a degli Studi di Pavia, 
		Via Ferrata, 27100 Pavia, Italy.
		E-mail: {\sf ngocmaimonica.huynh@unipv.it},
                {\sf luca.pavarino@unipv.it}.
		These authors were supported by grants of Istituto Nazionale di Alta Matematica (INDAM-GNCS), by the European High-Performance Computing Joint Undertaking EuroHPC under grant agreement No 955495 (MICROCARD), and by MUR (PRIN 202232A8AN\_002 and PRIN P2022B38NR\_001)
   funded by European Union - Next Generation EU.
		}
	\and
	Luca Franco Pavarino\footnotemark[1]
	\and
	Simone Scacchi\thanks{
	    Dipartimento di Matematica, Universit\`a degli Studi di Milano,
		Via Saldini 50, 20133 Milano, Italy.
		E-mail: {\sf simone.scacchi@unimi.it}.
		This author was supported by grants of Istituto Nazionale di Alta Matematica (INDAM-GNCS), by the European High-Performance Computing Joint Undertaking EuroHPC under grant agreement No 955495 (MICROCARD, and  by MUR (PRIN 202232A8AN\_003 and PRIN P2022B38NR\_002)
   funded by European Union - Next Generation EU.
		}
}

\begin{document}

\maketitle

\begin{abstract}
The aim of the present work is to design, analyze theoretically, and test numerically, a generalized Dryja-Smith-Widlund (GDSW) preconditioner for composite Discontinuous Galerkin discretizations of multicompartment parabolic reaction-diffusion equations, where the solution can exhibit natural discontinuities across the domain. We prove that the resulting preconditioned operator for the solution of the discrete system arising at each time step converges with a scalable and quasi-optimal upper bound for the condition number. 
The GDSW preconditioner is then applied to the EMI (Extracellular - Membrane - Intracellular) reaction-diffusion system, recently proposed to  model microscopically the spatiotemporal  evolution of cardiac bioelectrical potentials.
Numerical tests validate the scalability and quasi-optimality of the EMI-GDSW preconditioner, and investigate its robustness with respect to the time step size as well as jumps in the diffusion coefficients.  
\end{abstract}

\pagestyle{myheadings}
\markboth{N. M. M. Huynh, L. F. Pavarino, S. Scacchi}{GDSW preconditioners for composite DG discretizations of parabolic problems}

\section{Introduction}
In the present work, we construct and analyze a generalized Dryja-Smith-Widlund preconditioner for parabolic reaction-diffusion problems where the equations present a low-order term that can lead to discontinuities in the solution on the considered domain.
Such problems arises in many applications, such as microscopic modeling in biomechanics, porous media \cite{corti2023numerical, piersanti2021parameter}, and cardiac and excitable tissue modeling \cite{tveito2021tris, tveito2017, tveito2021}, where the low-order term either represents a 3D-1D (or 3D-0D) poroelasticity network or couples diffusive phenomena with microscopic ionic exchanges among cells. 
The numerical simulation of these phenomena very often presents challenges due to the different physics and scales involved, which require spatial and time discretizations that lead to matrix problems with very large dimensions.
Common choices for the space discretization are based on Discontinuous Galerkin (DG) methods (see \cite{corti2023discontinuous, corti2023numerical}), which usually need to be coupled with efficient preconditioned solvers in order to tackle their computational complexity.

In the context of DG discretizations, several solvers have been proposed, ranging from block-preconditioners \cite{kanschat2003preconditioning} (which can be tailored to each specific physics in case of coupled problems) to two-level methods for second order elliptic problems \cite{dobrev2006two}, as well as iterative and multilevel methods for elliptic equations with jumps coefficients \cite{ayuso2014multilevel, ayuso2009uniformly}.
An extensive study on the parallel performance of algebraic multigrid solvers on Graphics Processing Units (GPUs) can be found in \cite{centofanti2024multigrid},
In the Domain Decomposition (DD) framework, numerous studies have addressed both overlapping \cite{antonietti2007schwarz, antonietti2008multiplicative} and non-overlapping \cite{canuto2014bddc, dryja2013analysis, dryja2015deluxe} DD methods for DG discretizations of elliptic problems.
However, these works did not consider the specific challenges of reaction-diffusion problems with discontinuous solutions which are the focus of this work, where the discrete solutions must preserve the discontinuities of the continuous problem. 
Among the few works addressing this issue, we mention the
two-level algebraic multigrid methods \cite{budisa2023algebraic} and our previous work on dual-primal preconditioner, namely the Balancing Domain Decomposition by Constraints (BDDC) \cite{huynh2023convergence}. 
While these BDDC preconditioners have been proven to be mathematically solid and to perform well on structured meshes or whenever the local connectivity graphs of the nodes are simple or well-known \cite{chegini2023efficient, huynh2023convergence}, their implementation is not straightforward and high-performance libraries, such as PETSc \cite{balay2022petsc}, typically require specific matrix and vector structures, which are not easily accessible to intermediate users. Moreover, the reduced linear system to which the algebraic problem is recast is typically quite dense, which can pose challenges when simulating examples with computational domains obtained through image segmentation, where the number of degrees of freedom is high and their distribution complex.

In order to overcome this problem, we consider here a DD method which presents the advantage of a lighter non-overlapping coarse problem combined with the power and simplicity of overlapping Schwarz algorithms as local solvers. This method, known in the literature as generalized Dryja-Smith-Widlund (GDSW) preconditioner \cite{dohrmann2008domain}, has been largely studied and a parallel implementation can be found in \cite{heinlein2020frosch}. Several variants have been proposed by improving the construction of the coarse space \cite{heinlein2019adaptive, heinlein2022adaptive} and by multilevel extensions \cite{heinlein2023multilevel}.

In this paper, we extend the GDSW preconditioner to composite Discontinuous Galerkin
discretizations of multicompartment reaction-diffusion problems.
Since the abstract problem can be formulated as collection of many subproblems that interacts only on the interfaces of the related domain, there is a natural overlap between the construction of the GDSW subspaces and the subproblems definition, as we will show in the next Section.
The main result of this paper is the proof of a scalable and quasi-optimal upper bound for the condition number of the GDSW preconditioned operator for composite DG discretizations of reaction-diffusion problems. 
The GDSW preconditioner is then applied to the recently proposed EMI (Extracellular - Membrane - Intracellular) reaction-diffusion system, modeling microscopically the spatiotemporal  evolution of cardiac bioelectrical potentials.

The work is structured as follows: we start presenting a generalized problem formulation in Section \ref{sec: model}, along with a possible time discretization of its variational formulation. Section \ref{sec: preconditioner} introduces the standard GDSW preconditioner and the related coarse and local spaces. In particular, we propose a suitable formulation of the space subdivision.
The main convergence result is proven in Section \ref{sec: theoretical convergence} and confirmed numerically through extensive two-dimensional tests in Section \ref{sec: numerical results}. 
Several comments on further developments are given in the conclusive Section \ref{sec: conclusions}.

\section{A multicompartment parabolic reaction-diffusion problem}\label{sec: model}
{\bf The continuous model problem.} Given $N+1$ generic domains $\Omega_i \subset \mathbb R^d$ with $i=0,\dots,N$ and $d=2,3$, a time interval (0,T), such that the union of all the $\{\Omega_i\}_{i=0}^N$ domains forms a global domain $\Omega$ and such that $\Omega_i \cap \Omega_j = \emptyset$, with $i\neq j$, let us consider the following multicompartment parabolic reaction-diffusion problem: find $u=\{u_i\}_{i=0}^N$ such that it holds
\begin{equation} \label{eq: multiple parabolic pde}
    \begin{cases}
        - \dive (\sigma_i \grad u_i) = 0		&\qquad \text{in }\Omega_i, \ \ i=0,\dots, N,	\\
        - n_i^T \sigma_i \grad u_i = C_m \dfrac{\partial \jump{u}_{ij} }{\partial t} + F(\jump{u}_{ij})			&\qquad \text{on } \Eij = \overline{\Omega}_i \cap \overline{\Omega}_j \subset \partial \Omega_i, \; i\ne j,		\\
        n^T\sigma_i \grad u_i  = 0  & \qquad \text{on $\partial\Omega_i \cap \partial\Omega$}, \\
        u_i(0)=u_{i,0} &\qquad \text{in }\Omega_i, \ \ i=0,\dots, N.	\\
    \end{cases}
\end{equation}
Here, $F$ is the reaction term and  $\jump{u}_{ij} = u_i \cdot n_i - u_j \cdot n_j$ the jump between the value of $u_i$ and its neighboring value $u_j$ from the adjacent subdomain $\Omega_j$ along the common boundary $\Eij$, since $u_i$ is supposed to be discontinuous across the domains.

From an application viewpoint, this system could be seen as the union of (loosely) coupled problems: for instance, it can model the co-existence and interaction of different pollution agents in different regions over a time period; or it can represent the propagation of the electric potentials $u_i$ at a cellular scale and the time evolution of the coupling of several cellular dynamics within any organ tissue (being this the brain, the cardiac muscle, etc).

\begin{remark}
    We consider $N+1$ domains (or problems) since in the numerical tests we will consider a particular application case where we will benefit from this notation. Moreover, since the purpose of this work is to study the convergence rate of a preconditioner for a composite Discontinuous Galerkin discretization of (\ref{eq: multiple parabolic pde}), we consider the simplified two-dimensional setting of our model problem ($d=2$), leaving the three-dimensional case for further numerical studies.
\end{remark}

The weak formulation for the coupled problem on $\Omega$, by summing all the contributions from the $N+1$ domains, reads: 
given the initial data $u_{i,0}=u_i(0)$ for $i=0,\dots N$, find $u = \{ u_i \}_{i=0}^N$, $u_i \in L^2(0,T; H^1(\Omega_i))$, such that
\begin{equation}
\sumN \int_{\Omega_i} \sigma_i \grad u_i \grad \phi_i \, dx 	+ \dfrac{1}{2} \sumN \sumEij \int_{\Eij} \left( C_m \dfrac{\partial \jump{u}_{ij} }{\partial t} + F(\jump{u}_{ij}) \right) \jump{\phi}_{ij} \, ds = 0,
\label{weak_form}
\end{equation}
for suitable test functions $\phi_i \in H^1(\Omega_i)$, $i=0,\dots,N$.

{\bf Semidiscretization in time.} In order to ease the subsequent analysis, we consider an implicit-explicit (IMEX) time discretization, where the diffusion term is treated implicitly and the reaction explicitly. Alternative implicit time discretizations could be considered as well, analogously to the Bidomain implicit discretizations studies in \cite{huynh2022parallel,huynh2021newton}.

The time interval $(0,T)$ is subdivided into $K$ intervals and, by defining the time step $\tau = t^{k+1} - t^k$, the following scheme arises: for $k=0,\dots,K$, find 
$\{u^{k+1}_i\}_{i=0}^N$, with $u_i^{k+1} \in H^1(\Omega_i)$ such that  $\forall \phi_i \in H^1(\Omega_i)$, $i=0, \dots,N$

	\begin{equation*}
		\dfrac{1}{2} \sumN \sumEij \int_{\Eij} C_m \dfrac{\jump{u^{k+1}}_{ij} - \jump{u^k}_{ij} }{\tau} \jump{\phi}_{ij} \, ds \, + \sumN \int_{\Omega_i} \sigma_i \grad u_i^{k+1} \grad \phi_i \, dx 
		= -\dfrac{1}{2} \sumN \sumEij \int_{\Eij} F(\jump{u^k}_{ij}) \jump{\phi}_{ij} \, ds ,
	\end{equation*}
 or, equivalently, 
	\begin{multline*}
		\dfrac{1}{2} \sumN \sumEij \int_{\Eij} C_m \jump{u^{k+1}}_{ij} \jump{\phi}_{ij} \, ds + \tau \sumN \int_{\Omega_i} \sigma_i \grad u_i^{k+1} \grad \phi_i \, dx \\
		= \dfrac{1}{2} \sumN \sumEij \int_{\Eij} C_m \jump{u^{k}}_{ij} \jump{\phi}_{ij} \, ds - \tau \dfrac{1}{2} \sumN \sumEij \int_{\Eij} F(\jump{u^k}_{ij}) \jump{\phi}_{ij} \, ds.
	\end{multline*}
Let the above quantities be recast as follows:
	\begin{equation} \label{eq: a_i p_i}
		\begin{split}
			a_i (u, \phi) &:= \int_{\Omega_i} \sigma_i \grad u_i \grad \phi_i \, dx					\\
		      p_i (u, \phi) &:= \dfrac{1}{2} \sumEij \int_{\Eij} C_m \jump{u}_{ij} \jump{\phi}_{ij} \, ds	\\
		      f_i (\phi) &:= \dfrac{1}{2} \sumEij \int_{\Eij} \left( C_m \jump{u}_{ij} \jump{\phi}_{ij} - \tau F(\jump{u}_{ij}) \jump{\phi}_{ij} \right) \, ds 
		\end{split}
	\end{equation}
	\begin{equation} \label{eq: widetilde d_i}
        \widetilde d_i (u, \phi) := \tau a_i(u, \phi) + p_i(u, \phi).
	\end{equation}
 The semidiscrete problem at each time step can then be written in compact form as: find $u=\{u^{k+1}_i\}_{i=0}^N$, with $u_i^{k+1} \in H^1(\Omega_i)$ such that 
 
 \begin{equation}
     d(u, \phi) := \sumN \widetilde d_i (u, \phi) = f(\phi) := \sumN f_i(\phi) ,
     \quad \forall \phi=\{\phi_i\}_{i=0}^N, \quad \phi_i \in H^1(\Omega_i).
     \label{semidiscrete_pb}
 \end{equation}

{\bf The discrete problem.}  On each subdomain  $\Omega_i$, we consider a standard local finite element space $V^h(\Omega_i)$, and we denote by $V^h(\Omega)$ the product space of the $V^h(\Omega_i)$.
The fully discrete problem at each time step is then given by: find $u=\{u_i\}_{i=0}^{N} \in V^h(\Omega)$ such that 
\begin{equation} \label{eq:glob sys}
	d(u, \varphi) = f(\varphi),
		\qquad \forall\varphi=\{\varphi_i\}_{i=0}^{N} \in V^h(\Omega).
\end{equation}
Denoting by $A_i$ the local stiffness matrix associated with the bilinear form $a_i(\cdot, \cdot)$ and by $M_i$ the local mass matrix associated with the bilinear form $p_i(\cdot, \cdot)$, then (\ref{eq:glob sys}) can be written 
in matrix form as
	\begin{equation} \label{eq: global algebraic system}
	    \mathcal K \mathbf u = \mathbf f,
	    \qquad
	    \text{with } 
	    \qquad \mathcal K = \sum_{i=0}^N \mathcal K_i,
	    \qquad
	    \mathcal K_i = \tau A_i + M_i.
	\end{equation}
The right-hand side of the above system is updated at each time step, and the resulting linear system is solved by the Preconditioned Conjugate Gradient (PCG) method. We will design and analyze an efficient and scalable preconditioner in the next section.

\section{The GDSW preconditioner} \label{sec: preconditioner}
In the following, we consider and study the theoretical convergence of an overlapping Schwarz preconditioner, known in the literature as the generalized Dryja-Smith-Widlund (GDSW) preconditioner, modified to comply with the structure of our multicomparment reaction-diffusion problem (\ref{eq: multiple parabolic pde}), where the solution on $\Omega$ is allowed to be discontinuous across the different subdomains $\Omega_i$.
The idea behind the GDSW algorithm is to combine a coarse space based on iterative substructuring techniques with local components based on classical overlapping Schwarz techniques \cite{dohrmann2008domain}. \\ 

We recall that the global domain $\Omega$ is the union of non-overlapping (sub)domains $\Omega_i$, $i = 0, \dots, N$, and we have a finite element discretization on each subdomain. Our problem formulation allows also for non-matching discretizations across the subdomains (since the solution $u=\{u_i\}_{i=0}^N$ can be discontinuous), but for simplicity we analyze the case with matching nodes on the interface $\Gamma$ between the subdomains. 
In a two-dimensional situation, this interface is composed of edges $\edgeij$ and vertices $\vertexl$: the first are open subset of the interface $\Gamma$, containing all the degrees of freedom shared by two subdomains $\Omega_i$ and $\Omega_j$. The latter are endpoints of such edges and are usually shared by more than two subdomains.
In three-dimensions, the interface would consist of faces (two-dimensional entities shared by two subdomains), edges (one-dimensional entities shared by more than two subdomains) and vertices (endpoints of edges).

Let us then construct an overlapping subdomain partition $\{ \Omega_j' \}_j$ of $\Omega$, for $j=0,\dots,N$, where each $\Omega_j$ is extended by a strip of depth $\delta > 0$ (called overlap) of finite elements. 
Let us denote by $\Gamma_{j, \delta}$ this strip. Then $\Omega_j' = \Omega_j \cup \Gamma_{j,\delta}$. 
In the following and in the numerical tests, we will consider only the case of minimal overlap, i.e. $\delta = h$ (see Fig. \ref{fig:overlap} on the left).

\begin{figure}[!ht]
    \centering
    \includegraphics[scale=.5]{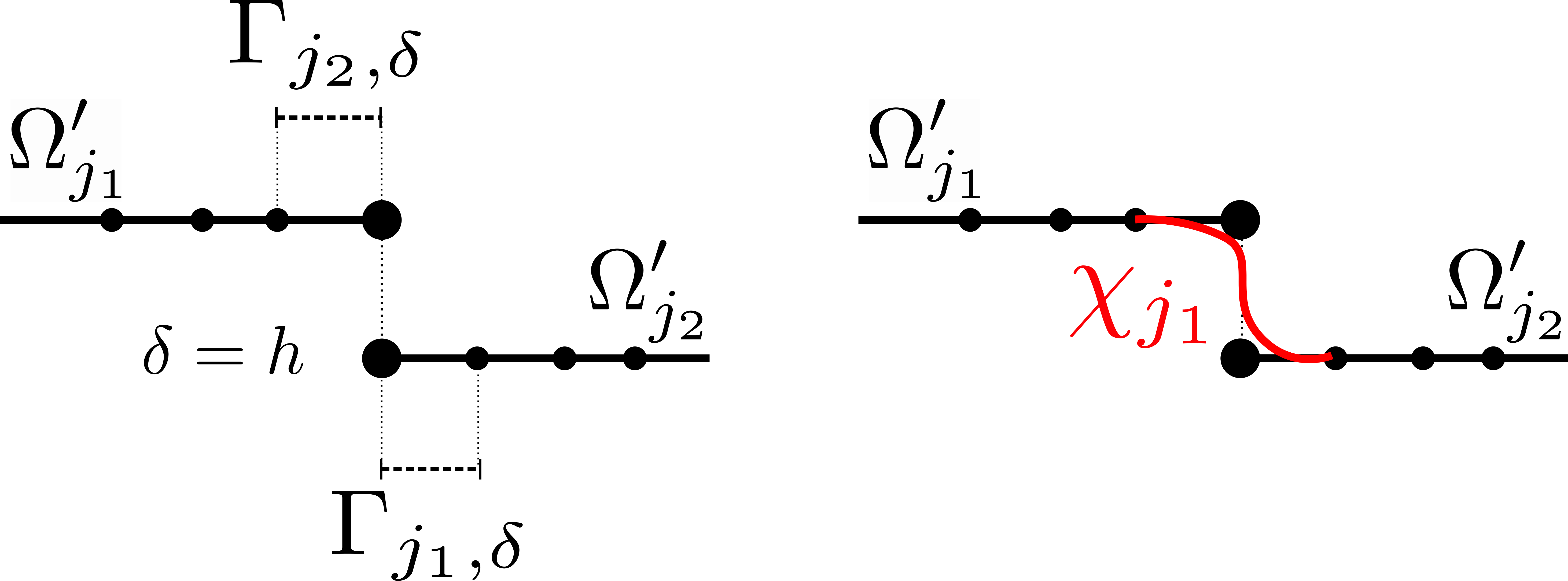}
    \caption{Overlap between two overlapping subdomains $\Omega'_{j_1}$ and $\Omega'_{j_2}$ for the multicompartment problem described in system (\ref{eq: multiple parabolic pde}), 1-dimensional example. On the left, representation of the current minimal overlapping situation. On the right, the considered partition of unity basis function $\chi_{j_1}$ for subdomain $\Omega'_{j_1}$.}
    \label{fig:overlap}
\end{figure}



\subsection{GDSW coarse space} \label{subsec: coarse space}
The coarse space $V_0^C$ is chosen as the space spanned by the vertex and edges functions, described below, extended as discrete harmonic functions inside each subdomain $\Omega_i$. These functions provide a partition of unity on the interface $\Gamma$. For each vertex $\vertexl$ we have as many vertex functions as the number of subdomains that share that vertex.
Denote by $\mathcal V^0_l$ the set of indices $k$ of subdomain $\Omega_k$ that shares vertex $\vertexl$, for all $l = 1, \dots, N_\text{vertices}$ (see Fig. \ref{fig:nonoverlap vertex}, left). Then the vertex functions associated with vertex $\mathcal V^l$ are given by
\begin{equation*}
    \basevertexl(x) = \{ \basevertexl_k(x) \}_{k \in \mathcal V^0_l},
\end{equation*}
where $\basevertexl_k$ is the basis function associated to $\mathcal V^l$ that has support in $\Omega_k$.
The edge functions $\baseedge{ij}(x)$ (represented in Fig. \ref{fig:nonoverlap vertex}, right) are instead unique for each edge $\Eij$.
\begin{figure}[!ht]
    \centering
    \includegraphics[scale=0.4]{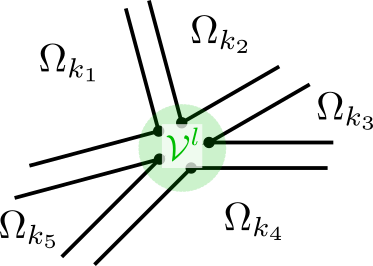}
    \hspace{20mm}
    \includegraphics[scale=0.4]{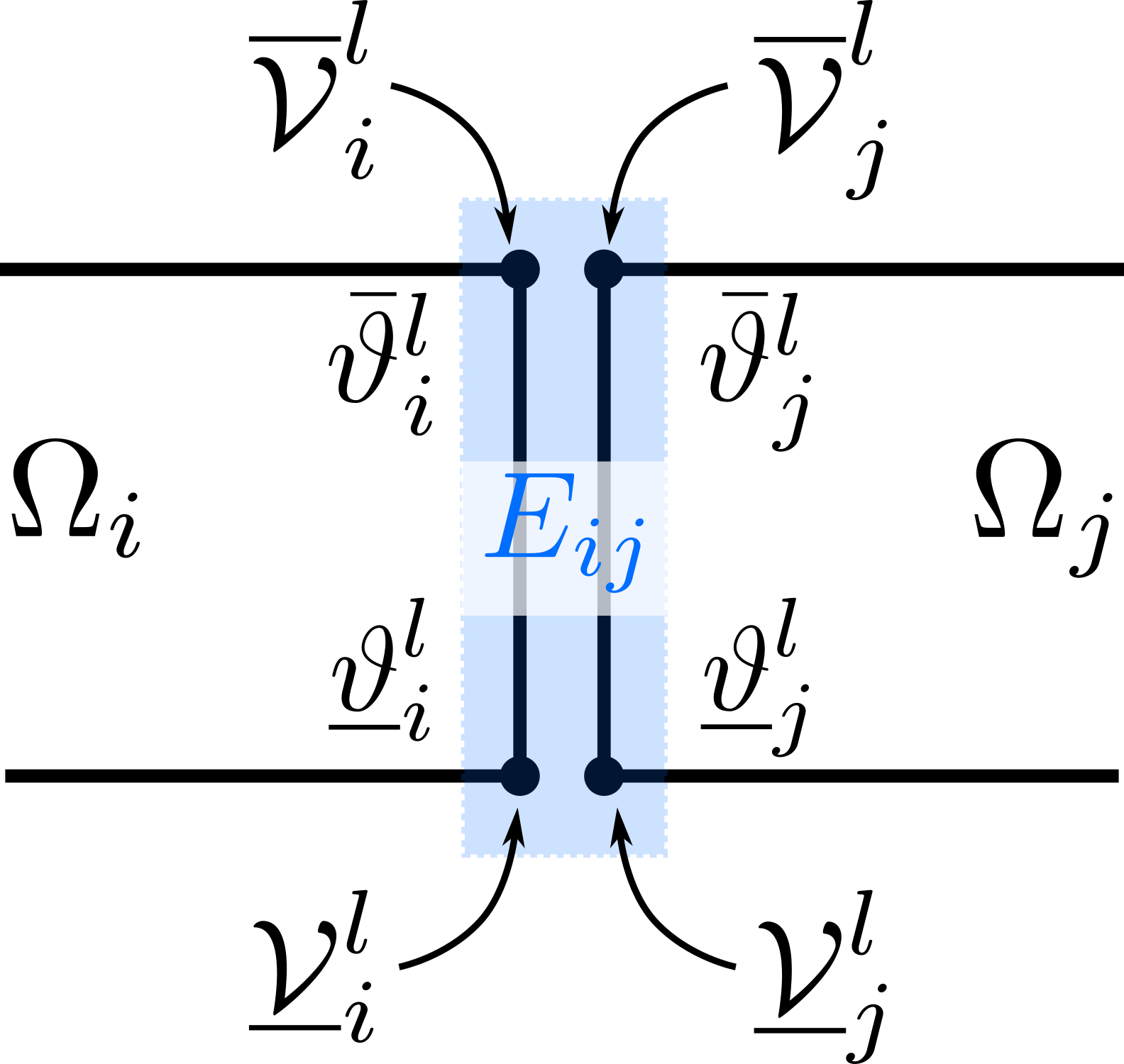}
    \caption{Vertex and edge sharing, 2-dimensional example. On the left, the vertex $\mathcal V^l$ is shared by five non-overlapping subdomains, therefore the set $\mathcal V^0_l$ contains five indices $\mathcal V^0_l = \{ k_1, k_2, k_3, k_4, k_5 \}$. On the right, the vertices $\undervertexl_i$, $\undervertexl_j$ and $\oververtexl_i$, $\oververtexl_j$ represent the same geometric endpoints of the edge $\Eij \subset \partial \Omega_i$, but are referred to different subdomains.}
    \label{fig:nonoverlap vertex}
\end{figure}

\noindent
Thus, the coarse component $u_0 \in V_0^C$ can be written as
\begin{equation}\label{eq: coarse decomposition}
    u_0 = \sum_{\vertexl} \sum_{k \in \mathcal V^0_l} u_k(\vertexl) \, \basevertexl(x) + \sumEij \avgedgeu{ij} \, \baseedge{ij} (x),
\end{equation}
where $\avgedgeu{ij}$ is the average of $u$ over the edge $\Eij$. The projection $P_0$ onto the coarse space $V_0^C$ is defined as
\begin{equation*}
    d(P_0 u, v) = d(u,v),
    \qquad
    \forall v \in V_0^C.
\end{equation*}

\subsection{GDSW local spaces} \label{subsec: local spaces}
We recall that the local non-overlapping finite element spaces have been denoted by $V_i = V^h (\Omega_i)$, while the product space can be trivially written as $V = V_0 \times \dots \times V_N$.
We then define the local space
\begin{equation*}
    V_i' = \{ v \in V : \, v = 0 \quad \text{in } \Omega \backslash \Omega_i' \} .
\end{equation*}
This space can also be written as
\begin{equation*}
    V_i' = V_i \times \{ \phi_j : \ \phi_j \neq 0 \text{ on } \partial \Omega_i, \, \phi_j \notin V_i \} .
\end{equation*}
Since $V_i' \subset V$, any function $u'_i \in V_i'$ can be written as 
$
     u_i' = \{ u'_{i,j} \}_{j=0}^N. 
$
Let $\chi_i$, $i=0,\dots, N$ be partition of unity as defined in \cite{dohrmann2008domain}.
Any function $u \in V$ can be written as 
\begin{equation*}
    u = u_0 + \sumNoverlap u_i',
    \qquad
    u_0 \in V_0^C, 
    \quad
    u_i' \in V_i', \text{ for } i = 1, \dots, N.
\end{equation*}
where $u_0 \in V_0^C$ has been defined in (\ref{eq: coarse decomposition}) and $u_i' \in V_i'$ are given by $u_i' = I^h (\chi_i (u - u_0) ) \in V_i' \subset V$, where $I^h$ interpolates into the product space $V^h(\Omega)$.
The projection-like operators $P_i$ onto the local space $V_i'$ are defined as $P_i = I_i\tilde{P}_i$, where
\begin{equation*}
    d(P_i u, v_i) = d(u, v_i),
    \qquad
    \forall v_i \in V_i'.
\end{equation*}
Our GDSW preconditioned operator is then defined as
\begin{equation}
    P_{ad} = P_0 + \sum_{i=1}^NP_i.
    \label{P_ad}
\end{equation}

\begin{remark}(Technical assumptions)
We will consider domains with the same geometrical properties as in Reference \cite{dohrmann2008domain}, in order to consider the same technical results stated in that work. The numerical tests presented in Section \ref{sec: numerical results} consider simpler and more trivial domains, and in that case some of the assumptions can be dropped and the convergence result may be sharpen.
Nevertheless, to enhance readability, we will not include here these technical tools, but we will refer to the specific Lemmas whenever required.	
\end{remark}

We prove here an auxiliary result that will be used in the proof of convergence rate bound; this is an adaptation of the bound (4.3) from \cite{dohrmann2008domain}.

\begin{lemma}\label{lemma 4.3}
    Let $u_0$ be the GDSW coarse function defined in (\ref{eq: coarse decomposition}). Then, for the multicompartment problem defined in (\ref{eq:glob sys}), it holds
    \begin{equation*}
        \| u - u_0 \|^2_{L^2(\Omega_i)} \leq C \left( 1 + \log \dfrac{H}{h} \right)^2 H^2 \Hseminorm{u}{i}^2,
    \end{equation*}
    with $H$ and $h$ the maximum subdomain and finite element diameters, respectively.
\end{lemma}
\begin{proof}
    By adding and subtracting the average $\avgdomainu{i}$ on the subdomain $\Omega_i$, it remains to bound only $\| u_0 - \avgdomainu{i} \|_{L^2(\Omega_i)}$, since $\| u - \avgdomainu{i} \|_{L^2(\Omega_i)}$ can be bound with the energy norm of $u$ by a Poincar{\'e} inequality type.
    The first term can be written as the sum of three parts,
    \begin{align}
        u_0 - \avgdomainu{i} &= \sum_{\vertexl} \left( u_i (\vertexl) - \avgdomainu{i} \right) \basevertexl_i (x) \label{eq: 4.3 line 1} \\
        &+ \sum_{\vertexl} \sum_{\substack{k \in \mathcal V^0_l \\k \neq i}} \left( u_k (\vertexl) - \avgdomainu{i} \right) \basevertexl_k (x) \label{eq: 4.3 line 2} \\
        &+ \sumEij \left( \avgedgeu{ij} - \avgdomainu{i} \right) \baseedge{ij} (x), \label{eq: 4.3 line 3} 
    \end{align}
    containing contributions from the vertices of $\Omega_i$ themselves (line \ref{eq: 4.3 line 1}), the adjacent vertices to each vertex of $\Omega_i$ (line \ref{eq: 4.3 line 2}) and the edges of $\Omega_i$ (line \ref{eq: 4.3 line 3}) respectively.
    The term (\ref{eq: 4.3 line 1}) can be bounded by using \cite[Lemma 3.2]{dohrmann2008domain} as follows:
    \begin{align*}
        \| \sum_{\vertexl} \left( u_i (\vertexl) - \avgdomainu{i} \right) \basevertexl_i (x) \|_{L^2(\Omega_i)}^2 &\leq \sum_{\vertexl} \int_{\Omega_i}  |u_i (\vertexl) - \avgdomainu{i} |^2 | \basevertexl_i (x) |^2 \ d\Omega \\
        &\leq C \left( 1 + \log \dfrac{H}{h} \right)^2 \Hseminorm{u}{i}^2 \sum_{\vertexl} \int_{\Omega_i}  | \basevertexl_i (x) |^2 \ d\Omega \\
        &\leq C \left( 1 + \log \dfrac{H}{h} \right)^2 H^2 \Hseminorm{u}{i}^2.
    \end{align*}
    The remaining terms (\ref{eq: 4.3 line 2}) and (\ref{eq: 4.3 line 3}) can be treated analogously.
\end{proof}
	
\section{Theoretical convergence rate bounds}\label{sec: theoretical convergence}
We are now ready to prove the convergence of GDSW preconditioners for the considered composite discontinuous Galerkin discretization of problem (\ref{eq: multiple parabolic pde}). 
\begin{theorem}
    The condition number of the GDSW preconditioned operator $P_{ad}$ defined in (\ref{P_ad}) for the discrete problem (\ref{eq:glob sys}) satisfies the bound
    \begin{equation*}
        \text{cond } (P_{ad}) \leq C \, \Phi_{\tau, H, h, \sigma_M, \sigma_m} \left( 1 + \log \dfrac{H}{h} \right)^2,
    \end{equation*}
    where $\Phi$ depends on the conductivity coefficients, the time step size $\tau$, the subdomain diameter $H$ and finite element size $h$
    \begin{equation*}
        \Phi_{\tau, H, h, \sigma_M, \sigma_m}  = \sigma_M \left( 1 + \dfrac{H}{h} \right) \left( 1 + \dfrac{1}{\tau \sigma_m} \right) 
        \qquad
        \sigma_m = \min_i |\sigma_i|,
        \qquad
        \sigma_M = \max_i |\sigma_i|,
    \end{equation*}
    and $C$ is a positive constant independent of $N, h, H, \tau$. 
 \label{main_theorem}
\end{theorem}
\begin{remark}
The dependence of $\Phi_{\tau, H, h, \sigma_M, \sigma_m}$ on the time step size $\tau$ is not observed in the results of the numerical tests in Sec. \ref{sec: numerical results}, indicating that the theoretical bound might not be sharp in $\tau$.
\end{remark}

\begin{proof}
The proof is based on general abstract Schwarz theory as well as previous works on GDSW preconditioners, see \cite{dohrmann2008domain, heinlein2019adaptive, heinlein2022adaptive, toselli2006domain}. In particular, we are required to verify three assumptions known as strengthened Cauchy-Schwarz inequality, local stability and stable decomposition, see (\cite{toselli2006domain}). However, by considering a standard coloring argument, the strengthened Cauchy-Schwarz inequality can be satisfied with a constant upper bound. Moreover, the local stability assumption holds true, since we use exact local solvers.
Therefore, we will only need to prove a stable decomposition for the considered subspace decomposition.

We proceed as in standard Schwarz theory by estimating the constant $C^2_0$ required by the stable decomposition \cite{toselli2006domain} and by considering the coarse and local solvers separately. \\

{\it Coarse solver.} As usually done in domain decomposition algorithms, we work with one subdomain $\Omega_i$ per time. We consider $u - \avgdomainu{i}$, and instead of (\ref{eq: coarse decomposition}) we consider $u_0 - \avgdomainu{i}$, being $\avgdomainu{i}$ the average of $u$ over subdomain $\Omega_i$:
    \begin{equation*}
        \avgdomainu{i}(x) = \avgdomainu{i} \left( \sum_{\vertexl} \sum_{k \in \mathcal V^0_l} \basevertexl_k (x) + \sumEij \baseedge{ij} \right).
    \end{equation*}
    Thus we want to bound the energy 
    \begin{equation*}
        \widetilde d_i (u_0 - \avgdomainu{i}, u_0 -\avgdomainu{i}) = \tau \, a_i (u_0 -\avgdomainu{i}, u_0 - \avgdomainu{i}) + p_i (u_0 - \avgdomainu{i}, u_0 - \avgdomainu{i}).
    \end{equation*}
    The first term can be treated in the following way: recalling the coarse decomposition (\ref{eq: coarse decomposition}), let us write explicitly $u_0 - \avgdomainu{i}$:
    \begin{equation*}
        u_0 - \avgdomainu{i} = 
        \sum_{\vertexl} \left( u_i(\vertexl) - \avgdomainu{i} \right) \basevertexl_i (x) 
        + \sum_{\vertexl} \sum_{\substack{k \in \mathcal V^0_l \\k \neq i}} \left( u_k(\vertexl) - \avgdomainu{i} \right) \basevertexl_k (x) 
        + \sumEij \left( \avgedgeu{ij} - \avgdomainu{i} \right) \baseedge{ij} (x).
    \end{equation*}
    The $H^1$-norm can be estimated by
    \begin{align*}
            a_i (u_0 - \avgdomainu{i}, &u_0 - \avgdomainu{i}) = \int_{\Omega_i} \sigma_i \grad ( u_0 - \avgdomainu{i} ) \cdot \grad ( u_0 - \avgdomainu{i} ) 
            \leq  \sum_{\vertexl} \int_{\Omega_i} \sigma_i \left( u_i(\vertexl) - \avgdomainu{i} \right)^2 | \grad \basevertexl_i (x) |^2 \\
            &+ \sum_{\vertexl} \sum_{\substack{k \in \mathcal V^0_l \\k \neq i}} \int_{\Omega_i} \sigma_i \left( u_k(\vertexl) - \avgdomainu{i} \right)^2 | \grad \basevertexl_k (x) |^2 
            + \sumEij \int_{\Omega_i} \sigma_i \left( \avgedgeu{ij} - \avgdomainu{i} \right)^2 | \grad \baseedge{ij} (x) |^2 \\
            &\overset{(A)}{\lesssim} | \sigma_i | \left( 1 + \log \dfrac{H}{h} \right) \Hseminorm{u}{i}^2 \left( \sum_{\vertexl} \Hseminorm{\basevertexl_i}{i}^2 + \sum_{\vertexl} \sum_{\substack{k \in \mathcal V^0_l \\k \neq i}} | \grad \basevertexl_k (x) |^2 + \sumEij \Hseminorm{\baseedge{ij}}{i}^2 \right) \\
            &\overset{(B)}{\lesssim} | \sigma_i | \left( 1 + \log \dfrac{H}{h} \right)^2 \Hseminorm{u}{i}^2 
            \leq \left( 1 + \log \dfrac{H}{h} \right)^2 a_i (u,u)
    \end{align*}
    For inequality $(A)$, the terms $u_k(\vertexl) - \avgdomainu{i}$ and $\avgedgeu{ij} - \avgdomainu{i}$ can be bounded using \cite[Lemma 3.2]{dohrmann2008domain}, while $(B)$ is obtained thanks to the fact that the first seminorm in the brackets is limited and the second is zero by definition, while the last term can be bounded considering \cite[Lemma 3.4]{dohrmann2008domain}.\\

    \noindent
    Regarding the jump term $p_i(\cdot, \cdot)$, it is useful to write explicitly the average $\avgdomainu{i}$ over the edge $\Eij$ in terms of edge basis functions (see Figure \ref{fig:nonoverlap vertex} for the notation):
    \begin{equation*}
        \avgdomainu{i} (x) =
        \begin{cases} 
            \avgdomainu{i} \left( \overbasevertexl_i (x) + \baseedge{ij} (x) + \underbasevertexl_i (x) \right) \\
            \avgdomainu{i} \left( \overbasevertexl_j (x) + \baseedge{ij} (x) + \underbasevertexl_j (x) \right)
        \end{cases}
    \end{equation*}
    We need to focus on the jump
    \begin{equation*}
        \jump{u_0 - \avgdomainu{i}}_{ij} = \left( u_{0,i} - \avgdomainu{i} \right) - \left( u_{0,j} - \avgdomainu{i} \right),
    \end{equation*}
    being
    \begin{equation*}
        \begin{aligned}
            u_{0,i} &= u (\oververtexl_i) \, \overbasevertexl_i (x) + u (\undervertexl_i) \, \underbasevertexl_i (x) + \avgedgeu{ij} \baseedge{ij} (x) \\
            u_{0,j} &= u (\oververtexl_j) \, \overbasevertexl_j (x) + u (\undervertexl_j) \, \underbasevertexl_j (x) + \avgedgeu{ij} \baseedge{ij} (x)
        \end{aligned}
    \end{equation*}
    The inner edge contribution $\avgedgeu{ij} \baseedge{ij}(x)$ can be neglected, since it appears from both sides of $\Eij$. 
    Thus, if we focus on only one endpoint of $\Eij$ (for instance $\oververtexl$), we need to estimate the jump of $\bar w = ( u(\overbasevertexl_i) - \avgdomainu{i} ) \, \overbasevertexl_i (x) - ( u(\overbasevertexl_j) - \avgdomainu{i} ) \, \overbasevertexl_j (x)$:
    \begin{equation*}
        \begin{aligned}
            p_i (\bar w, \bar w) &\lesssim \sumEij \left( 1 + \log \dfrac{H}{h} \right) \Hseminorm{u}{i}^2 \left( \| \overbasevertexl_i (x) \|_{L^2(\Eij)}^2 + \| \overbasevertexl_j (x) \|_{L^2(\Eij)}^2 \right) \\
            &\lesssim \left( 1 + \log \dfrac{H}{h} \right) \Hseminorm{u}{i}^2 
            \lesssim \dfrac{1}{| \sigma_i |} \left( 1 + \log \dfrac{H}{h} \right) a_i (u,u)
        \end{aligned}
    \end{equation*}
    In the same fashion, it is possible to estimate the jump for the term $\underline w = ( u(\underbasevertexl_i) - \avgdomainu{i} ) \, \underbasevertexl_i (x) - ( u(\underbasevertexl_j) - \avgdomainu{i} ) \, \underbasevertexl_j (x)$.
    
    We can estimate the energy related to the $i$-th problem by
    \begin{align*}
        \widetilde d_i (u_0 - \avgdomainu{i}, u_0 -\avgdomainu{i}) 
        &= \tau \, a_i (u_0 -\avgdomainu{i}, u_0 - \avgdomainu{i}) + p_i (u_0 - \avgdomainu{i}, u_0 - \avgdomainu{i}) \\
        &\leq C \left( 1 + \dfrac{1}{\tau |\sigma_i|} \right) \left( 1 + \log \dfrac{H}{h} \right)^2 \widetilde d_i (u,u),
    \end{align*}
    thus leading to
    \begin{equation*}
        d (u_0 - \avgdomainu{i}, u_0 -\avgdomainu{i}) \leq C \left( \tau + \dfrac{1}{\sigma_m} \right) \left( 1 + \log \dfrac{H}{h} \right)^2 d (u,u),
    \end{equation*}
    where $\sigma_m = \min_{i=0,\dots,N} |\sigma_i|$.\\

{\it Local solvers.} In Section \ref{subsec: local spaces} we have defined 
    \begin{equation*}
        u = u_0 + \sumNoverlap u_i', \quad \text{for } u_0 \in V_0^C, 
        \qquad
        \text{and }
        \qquad
        u_i' = I^h ( \chi_i (u-u_0)) =: I^h \chi_i w \in V_i',
        \quad 
        \text{with } i=0,\dots,N. 
    \end{equation*}
    We recall that, for all $u_i'$, $v_i' \in V_i'$, $u'_i = \{ u'_{i,j} \}_{j=0}^N $,
    \begin{align}
        d_i (u_i', v_i') := d(u_i', v_i') &= \tau \sum_{k=0}^N \sigma_k \int_{\Omega_k} \grad u'_{i,k} \grad v'_{i,k} + \sum_{E_{kj}} C_m \int_{E_{kj}} (u'_{i,k} - u'_{i,j} ) (v'_{i,k} - v'_{i,j} ) \nonumber \\
        &= \tau \sigma_i \int_{\Omega_i} \grad u'_{i,i} \grad v'_{i,i} + \tau \sum_{j \, : \,  \Omega_j \cap \Gamma_{i,\delta} \neq \emptyset} \sigma_j \int_{\Omega_j \cap \Gamma_{i,\delta}} \grad u'_{i,j} \grad v'_{i,j} \label{eq: stiff local} \\
        &+ \sum_{j \, : \, \partial \Omega_i \cup \partial \Omega_j \neq \emptyset} C_m \int_{\Eij} (u'_{i,i} - u'_{i,j} ) (v'_{i,i} - v'_{i,j} ). \label{eq: jump local}
    \end{align}
    Here the same considerations made in the proof of \cite[Theorem 3.1]{dohrmann2008domain} hold. 
    Thus, the energy norms in line (\ref{eq: stiff local}) (consider for instance the first term - the second follows analogously) can be bounded as in the proof of \cite[Lemma 3.10]{toselli2006domain}:
    \begin{align*}
        \Hseminorm{u'_{i,i}}{i}^2 &\leq \Hseminorm{I^h ( \chi_i w)}{i}^2 
        \leq \dfrac{C}{\delta^2} \left[ \delta^2 \Hseminorm{w}{i}^2 + \dfrac{H}{\delta} \delta^2 \Hnorm{w}{i}^2 \right] \\
        &\leq \dfrac{C}{\delta^2} \left[ \delta^2 \left( 1 + \dfrac{H}{\delta} \right) \Hseminorm{w}{i}^2 + \dfrac{H}{\delta} \delta^2 \parallel w \parallel^2_{L^2(\Omega_i)}  \right] ,
    \end{align*}
    where $\delta$ is the common overlap parameter, which we chose in Section \ref{sec: preconditioner} to be $h$.
    Each of the above contributions can be bound as follows.
    \begin{itemize}
        \item[i)] By considering the triangle inequality and the result obtained in the coarse solver for the energy norm, it holds
        $$\Hseminorm{w}{i}^2 \lesssim \left( 1 + \log \dfrac{H}{h} \right)^2 \Hseminorm{u}{i}^2.  $$
        \item[ii)] Thanks to Lemma \ref{lemma 4.3}, we have
        $$\parallel w \parallel^2_{L^2(\Omega_i)} \leq C \left( 1 + \log \dfrac{H}{h} \right)^2 H^2 \Hseminorm{u}{i}^2.$$
    \end{itemize}
    Collecting the above contributions leads to the following bound for the energy norms of line (\ref{eq: stiff local}):
    \begin{align*}
        |\sigma_i| \Hseminorm{u'_{i,i}}{i}^2 &\leq C |\sigma_i | \left( 1 + \log \dfrac{H}{h} \right)^2 \left[ \Hseminorm{u}{i}^2 + \dfrac{H}{\delta} \Hseminorm{u}{i}^2 + \dfrac{H}{\delta} H^2 \Hseminorm{u}{i}^2 \right] \\
        &\leq C |\sigma_i | \left( 1 + \log \dfrac{H}{h} \right)^2 \left( 1 + \dfrac{H}{\delta} \right) \Hseminorm{u}{i}^2 .
    \end{align*}
    The cubic term in $H$ can be neglected, since it vanishes when the diameter $H$ decreases by increasing the number of subdomains $\Omega_i$ and keeping fixed the global domain $\Omega$.  
    Regarding the jump contributions in line (\ref{eq: jump local}), it holds
    \begin{equation*}
        p_i (u'_i, u'_i) = p_i(I^h \chi_i w, I^h \chi_i w) \lesssim p_i (u, u) + p_i (u_0, u_0).
    \end{equation*}
    The first term appears in the bilinear form $d_i(\cdot, \cdot)$, while the bound for the second term follows the same proof procedure as in the coarse solver by adding and subtracting the average $\avgdomainu{i}$.

    Therefore, with the choice of $\delta = h$ (see Section \ref{sec: preconditioner}), the local solvers carry the bound
    \begin{equation*}
        d_i(u'_i, u'_i) = d_i (I^h \chi_i w, I^h \chi_i w) \leq C \sigma_M \left( 1 + \dfrac{H}{h} \right) \left( 1 + \log \dfrac{H}{h} \right)^2 d_i(u,u),
    \end{equation*}
    where $\sigma_M = \max_{i=0,\dots,N} |\sigma_i|$.

    In conclusion, by collecting the above result and the coarse solver estimate, 
    \begin{equation*}
        d(u_0, u_0) + \sumNoverlap d_i (u_i', u_i') \leq C \sigma_M \left( 1 + \dfrac{H}{h} \right) \left( 1 + \dfrac{1}{\tau \sigma_m} \right) \left( 1 + \log \dfrac{H}{h} \right)^2 d(u,u).
    \end{equation*}
\end{proof}

\section{An application to the cardiac EMI (Extracellular - Membrane - Intracellular) reaction - diffusion model} \label{sec: numerical results}
We consider a particular application from the electrophysiology field for numerical testings of the proposed preconditioner. Our focus is the so-called cardiac EMI model\footnote{The acronym EMI stands for Extracellular, cell Membrane and Intracellular spaces, since this formulation takes into account each of these as separate entities.} \cite{tveito2021bis, tveito2021tris, tveito2017, tveito2021}, which provides a microscopic representation of the electrical propagation in the cardiac tissue by means of diffusion equations for each single cell coupled through a (possibly non-linear) reaction term on the boundaries (cell membranes).  
This innovative model overcomes the well-established cardiac Bidomain model \cite{franzone2014mathematical}, since it represents the extracellular space as well as the intracellular space and the cell membrane as individuals, allowing for realistic characterization of each. While for the Bidomain equations we can find several works related to its mathematical properties \cite{franzone2014mathematical, fedele2023comprehensive, tung1978bidomain, veneroni2009reaction} and solution strategies \cite{africa2023matrix, barnafi2024robust, centofanti2023comparison, chegini2022efficient, huynh2022parallel, huynh2021newton, plank2007algebraic, vigmond2008solvers, weiser_2022}, these various aspects for the cardiac EMI model are still open problems.

\subsection{The EMI microscopic description of cardiac electrical propagation}
We consider $N$ connected cells immersed in the extracellular liquid, which altogether form the cardiac tissue $\Omega$, where generally $\Omega\subset\R^d$, with $d\in\{2,3\}$; in the following numerical experiments we will consider $d=2$, in compliance with the proposed convergence analysis.
It is straightforward to visualize the parallelism between these $N+1$ objects and the non-overlapping partition $\{ \Omega_i \}_i$, with $i=0, \dots, N$ (by denoting with $\Omega_0$ the extracellular subdomain).
The interaction between each cells, the extracellular media and their neighbouring happens by means of ionic exchanges, which provide the reaction term on the boundaries $\partial \Omega_i$. These currents are allowed to propagate among the intracellular spaces through the gap junctions, special protein channels which allow the passage of ions directly between two intracellular environments, \cite{rohr2004}.

\begin{figure}
    \centering
    \includegraphics[scale=.45]{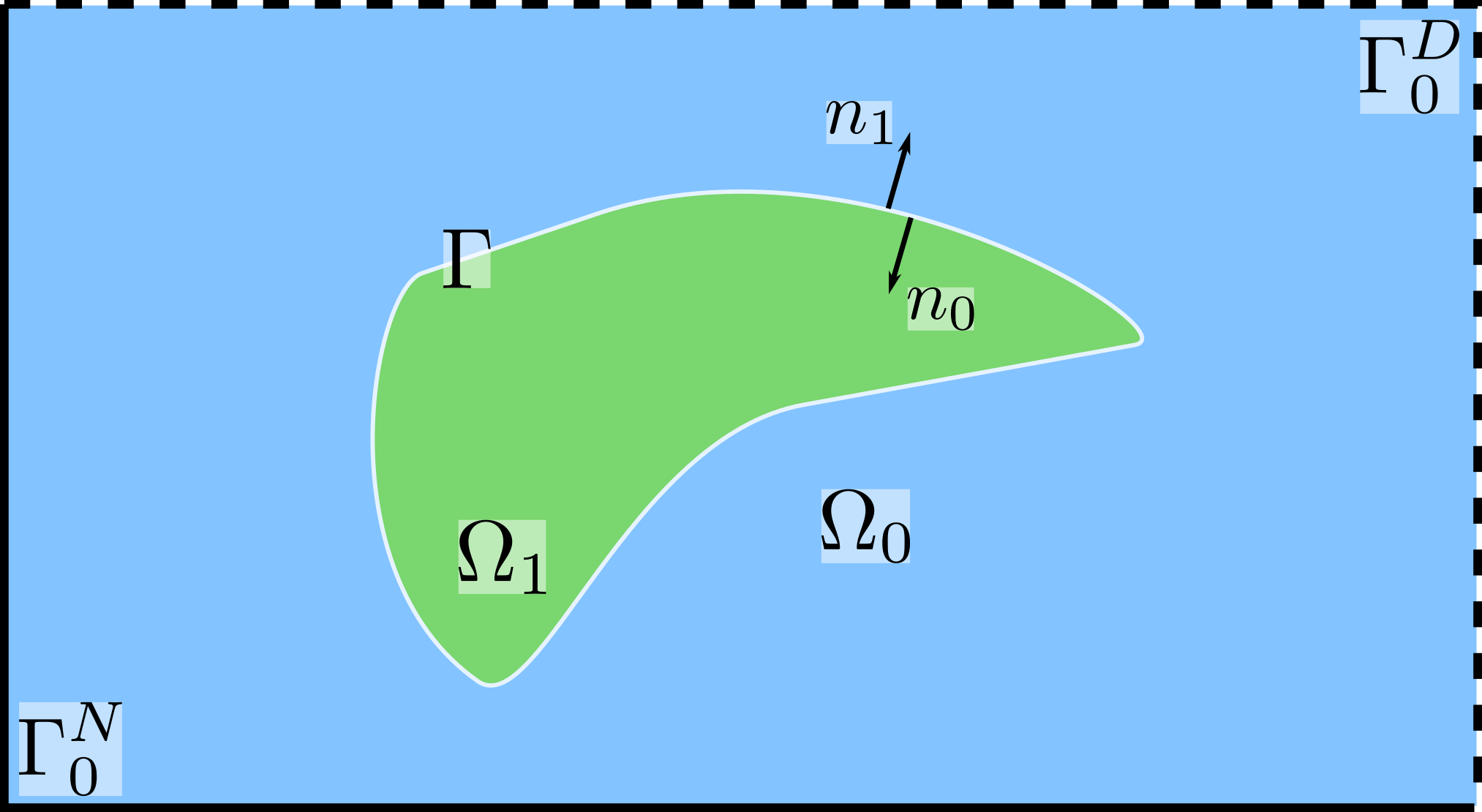}
    \hspace{5mm}
    \includegraphics[scale=.45]{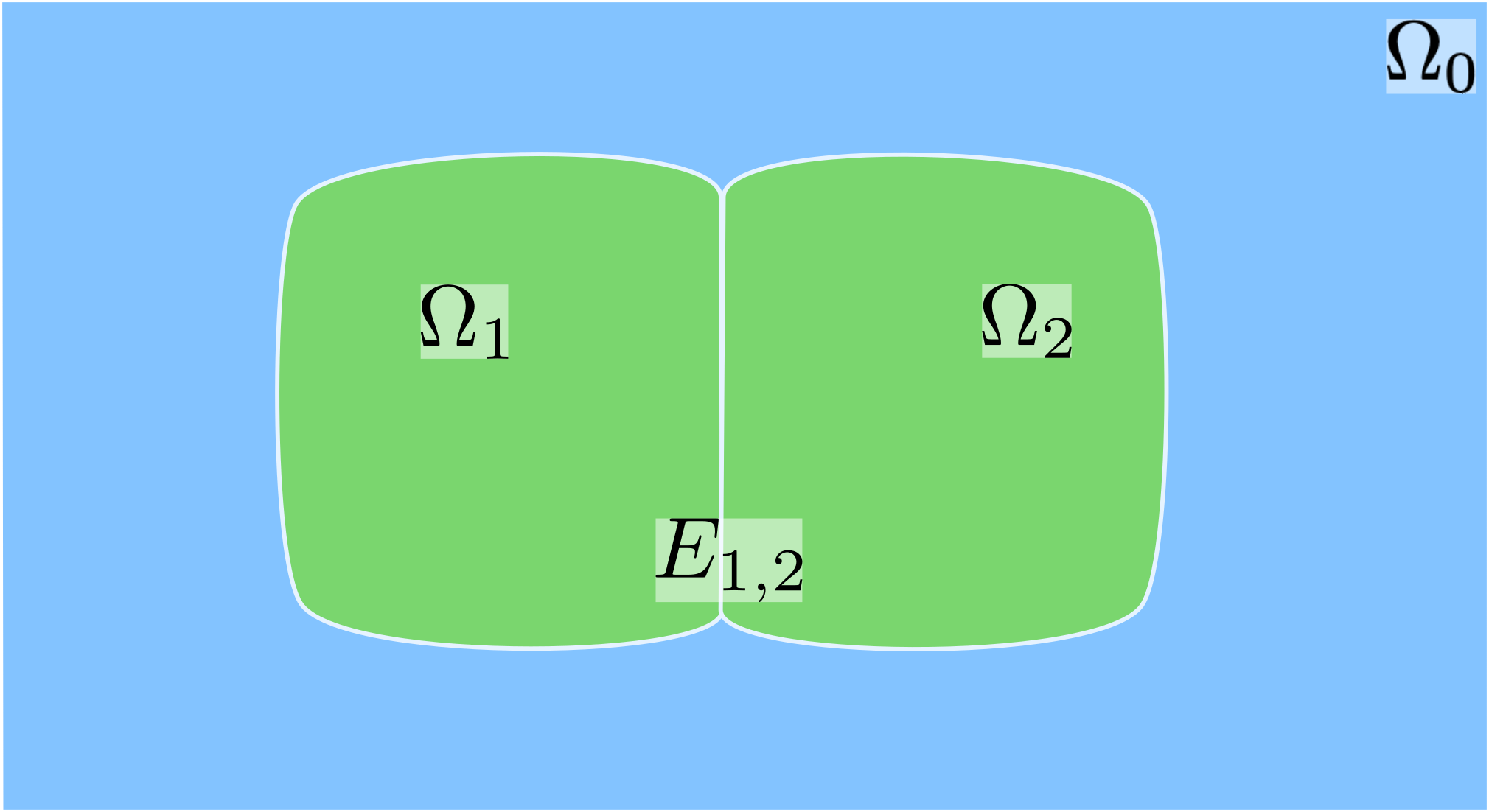}
    \caption{Left: representation of the situation described in system (\ref{eq: multiple parabolic pde}), with only one cell $\Omega_1$ (green) immersed in the extracellular liquid (light blue); the external boundary of the extracellular space $\Omega_0$ is divided into $\Gamma_0^D$ (black, dashed) and $\Gamma_0^N$ (black, solid), with boundary conditions given in (\ref{eq: external boundary conditions}). Right: representation of the situation described in system (\ref{eq: multiple parabolic pde}) considering two neighbouring cells $\Omega_1$ and $\Omega_2$, with common boundary $E_{1,2}$.}
    \label{fig: cells}
\end{figure}

The situation depicted on the left of Figure \ref{fig: cells} can be modeled by system (\ref{eq: multiple parabolic pde}), where we need to add the ionic equations
\begin{equation} \label{eq: ionic model}
   \dfrac{\partial c}{\partial t} - C(\jump{u}_{ij}, w, c) = 0, 
	\qquad 
	\dfrac{\partial w}{\partial t} - R(\jump{u}_{ij}, w) = 0,  
\end{equation}
that model the ion flow dynamic by means of ordinary differential equations, describing the time evolution of ion concentrations $c$ and gating variables $w$. 
Here, the transmembrane voltage $\jump{u}_{ij} = u_i - u_j$ represents the jump in the value of the electric potentials between two neighboring cells $i$ and $j$.
Thus, the reaction term $F(\cdot)$ will depend also from the gating and concentration variables, representing either the ionic current $I_\text{ion}(\jump{u}_{ij}, c, w)$ or the gap junctions $G(\jump{u}_{ij})$ (which we assume here to be linear in the potential jump). 
We assume that the (extracellular) potential is fixed on part of the external boundary $\Gamma^D_0$ while the remaining $\Gamma^N_0 = \Gamma \backslash \Gamma^D_0$ is insulated:
\begin{equation} \label{eq: external boundary conditions}
    u_0 = 0 \quad \text{on } \Gamma^D_0 ,
    \qquad 
    -n^T_0 \cdot \sigma_0 \grad u_0 = 0 \quad \text{on } \Gamma^N_0.
\end{equation}

In this framework, $\sigma_i$ is the conductivity coefficient\footnote{In general $\sigma_i$ are tensors; however, in this work, we have treated them as scalar since the EMI model assumes isotropic diffusion in the cells and in the extracellular matrix. This is motivated by the fact that the pronounced anisotropy in the homogenized bidomain model is an effect of the cellular geometry, which here in the EMI model is resolved explicitly.} in $\Omega_i$ and  $n^T_i$ the outward normal on $\partial \Omega_i$. We refer to \cite{tveito2021bis, tveito2017} for a formal derivation of the complete EMI system.

The solution obtained from such problem is defined up to an arbitrary constant: in order to ensure uniqueness from a numerical point of view, we impose zero average over its domain for the extracellular component $u_0$. 

In conclusion, the cardiac EMI model reads
\begin{equation} \label{eq: EMI model}
    \begin{cases}
        - \dive  (\sigma_i \grad u_i) = 0  & \text{in } \Omega_i \quad \forall i = 0, \dots, N, \\
        u_i - u_j = \jump{u}_{ij}       & \text{on } \Eij, \\
        - n^T_i \sigma_i \grad u_i = C_m \dfrac{\partial \jump{u}_{ij}}{\partial t} + F(\jump{u}_{ij}, c, w)  & \text{on } \Eij, \\
	    \dfrac{\partial c}{\partial t} - C(\jump{u}_{ij}, w, c) = 0, 
		\qquad 
		\dfrac{\partial w}{\partial t} - R(\jump{u}_{ij}, w) = 0,\\
        u_i(0) = u_{i,0}, \qquad w(0)=w_0, \qquad c(0)=c_0, & \text{in } \Omega_i \quad \forall i = 0, \dots, N. \\
    \end{cases}
\end{equation}
	 
We consider a splitting strategy for the time solution of system (\ref{eq: EMI model}). At each time step, we solve first the ionic model, given the jump $\jump{u}_{ij}$ from the previous time step; then, we update the cell-by-cell model with the newly-computed $c$ and $w$ and solve it with respect to the electric potential. 
With this approach, we can easily refer to the discretized formulation derived in Section \ref{sec: model} and the preconditioning technique introduced in Sec. (\ref{sec: preconditioner}), obtaining an EMI-GDSW preconditioner. 
Another recently proposed approach consider a boundary integral formulation of system (\ref{eq: EMI model}), see \cite{rosilho2024boundary}.

\subsection{Numerical tests}
We consider a Matlab implementation of our EMI-GDSW preconditioner for a two-dimensional rectangular geometry where the extracellular domain frames a group of cells. This code considers a linear gap junction between cells, each coinciding with an edge $\Eij$, and the Aliev-Panfilov ionic model \cite{aliev1996simple} for the update of the gating variables; this model does not include concentration variables. 
We fix the time step size to $\tau = 0.05$ ms, except for the last Section \ref{sec: dependence tau}, where the parameter $\tau$ is varied between $0.005$ and $0.1$ ms in order to study its effects on the EMI-GDSW convergence rate. 
The external current needed for the activation is applied at the bottom-left corner of the domain for $1$ ms, with an intensity of 50 mA/cm$^2$; the total simulation time is of $[0,5]$ ms. The initial value for the potential $u_i(0)$ is set to $-85$ mV, while the initial gating $w_0$ value is set to 0.
Unless otherwise specified, the conductivity coefficients are  $\sigma_i = 3 \times 10^{-3}$ in all cells.

The discrete linear system (\ref{eq: global algebraic system}) arising at each time step is solved through an iterative Conjugate Gradient method, either unpreconditioned (CG) or with a preconditioning strategy. The stopping criterion compares the $L^2$-norm of the relative (preconditioned) residual with a fixed tolerance of $10^{-6}$.
For simplicity, we implemented a GDSW preconditioner with a coarse problem spanned by the subdomain vertex basis functions, see Sec. (\ref{subsec: coarse space}).
For comparison,  
we also consider a classic Additive Schwarz algorithm (AS); see \cite{toselli2006domain} for further details. 

All the tests have been performed on a Linux workstation equipped with an Intel i9-10980XE CPU with 18 cores running at 3.00 GHz.

\subsubsection{Scalability tests}
We start with numerical results investigating the scalability of the proposed EMI-GDSW preconditioned solver. 
We consider a time interval of $[0,5]$ ms, for a total of 100 uniform time steps.
The number of cells (subdomains) considered varies from $N = 2\times 2$ to $N = 32 \times 32$, each one discretized by $24 \times 4$ finite elements.
We report the condition number ($k_2$) and the number of linear iterations (it) at the final time step of the simulation; see Table \ref{tab: scal} and Figure \ref{fig: scal}.
The results show that the GDSW preconditioner has the best performance in all the cases considered, in terms of both number of linear iterations and condition numbers. Both these parameters mildly increase while initially increasing the number of cells $N$, but they approach a constant upper bound afterwards, in agreement with the main bound of Theorem \ref{main_theorem}, since the ratio $H/h$ is kept constant in this test. On the other hand, for both unpreconditioned CG as well as AS, iteration counts and condition numbers increase with increasing $N$.

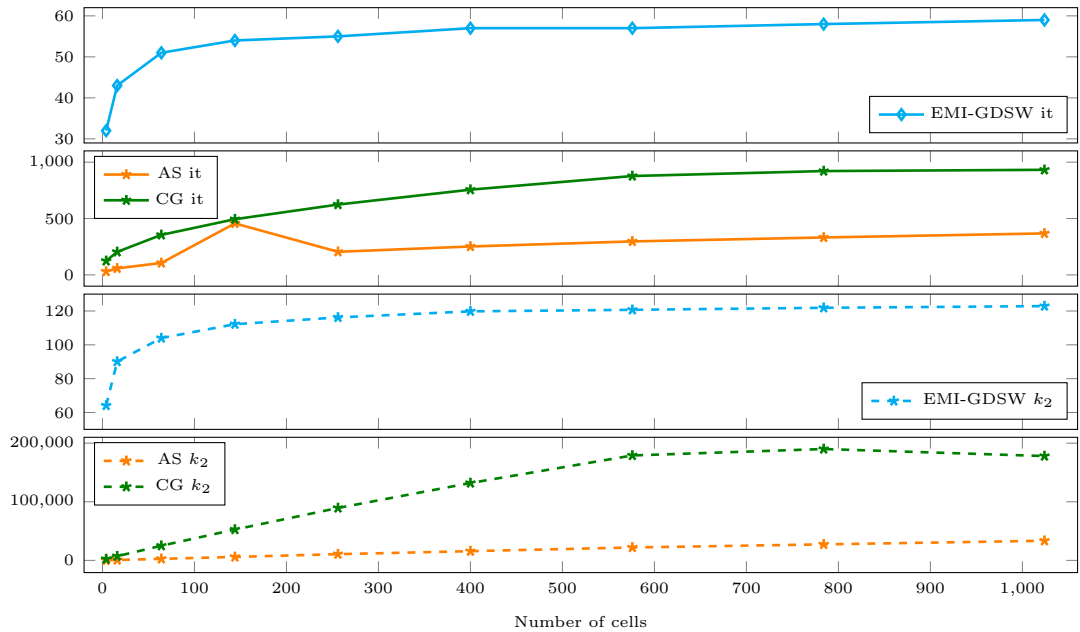
\begin{figure}
	    \centering
\begin{subfigure}{\textwidth}
   \centering
	    \begin{tabular}{c|rcrcrc}
	        \toprule
	        \multirow{2}{*}{nb of cells} &\multicolumn{2}{c}{EMI-GDSW}    &\multicolumn{2}{c}{AS}     &\multicolumn{2}{c}{CG}    \\
	                        &$k_2$ &it      &$k_2$ &it         &$k_2$  &it       \\      
	        \midrule
	        $2\times2$      &64.1   &32       &149.2       &30       &2.20e+03   &124   \\
	        $4\times4$      &90.1   &43       &632.4       &58       &7.09e+03   &204   \\
	        $8\times8$      &104.0  &51       &2.61e+03   &106       &2.50e+04   &355   \\
	        $12\times12$    &112.2  &54       &5.87e+03   &157       &5.26e+04   &494   \\
	        $16\times16$    &116.2  &55       &1.03e+04   &205       &8.94e+04   &624   \\
	        $20\times20$    &119.8  &57       &1.58e+04   &252       &1.32e+05   &756   \\
	        $24\times24$    &120.7  &57       &2.20e+04   &297       &1.79e+05   &877   \\
	        $28\times28$    &121.9  &58       &2.73e+04   &332       &1.90e+05   &921   \\
	        $32\times32$    &122.9  &59       &3.35e+04   &368       &1.78e+05   &932   \\
	        \bottomrule
	    \end{tabular}
     \caption{Condition number ($k_2$)  and linear iterations (it) at final time $t=5$ ms for EMI-GDSW (left), AS (center), unpreconditioned CG (right).}
     	    \label{tab: scal}
\vspace{0.7cm}
 \end{subfigure}
\begin{subfigure}{\textwidth}
	\begin{tikzpicture}
		\begin{groupplot}
			[
			group style={group size=1 by 4, vertical sep=3pt, ylabels at=edge left,
				x descriptions at=edge bottom},
			width=.8\textwidth,
			height=1.8cm,
			scale only axis,
			scaled ticks = false,
			tick label style={/pgf/number format/fixed, /pgf/number format/precision=5},
			tick label style={font=\tiny},
			label style={font=\tiny},
			legend style={font=\tiny},
			xlabel={Number of cells},
			xmin=-20, xmax=1060,
			unbounded coords=jump]
\nextgroupplot[legend style={at={(0.99,0.21)},anchor=east}, ymin=29, ymax=62]
            \addplot+[cyan, line width=1pt, mark=diamond] table [x=cells, y=iter_gdsw, col sep=comma, filter discard warning=false, unbounded coords=discard] {plot/results/numer.csv};			
			\legend{EMI-GDSW it};	
			\nextgroupplot[legend style={at={(0.01,0.73)},anchor=west}, ymin=-100, ymax=1100]
			\addplot+[orange, line width=1pt, mark=star] table [x=cells, y=iter_as, col sep=comma, filter discard warning=false, unbounded coords=discard] {plot/results/numer.csv};
			\addplot+[green!50!black, line width=1pt, mark=star] table [x=cells, y=iter_cg, col sep=comma, filter discard warning=false, unbounded coords=discard] {plot/results/numer.csv};
			\legend{AS it, CG it};
			\nextgroupplot[legend style={at={(0.99,0.21)},anchor=east}, ymin=50, ymax=130]
			\addplot+[cyan, dashed, line width=1pt, mark=star] table [x=cells, y=cond_gdsw, col sep=comma, filter discard warning=false, unbounded coords=discard] {plot/results/numer.csv};
			\legend{EMI-GDSW $k_2$};
			\nextgroupplot[legend style={at={(0.01,0.73)},anchor=west}, ymax=210000]
			\addplot+[orange, dashed, line width=1pt, mark=star] table [x=cells, y=cond_as, col sep=comma, filter discard warning=false, unbounded coords=discard] {plot/results/numer.csv};
			\addplot+[green!50!black, dashed, line width=1pt, mark=star] table [x=cells, y=cond_cg, col sep=comma, filter discard warning=false, unbounded coords=discard] {plot/results/numer.csv};
			\legend{AS $k_2$, CG $k_2$};
		\end{groupplot}
	\end{tikzpicture}
 \caption{Plots of condition number ($k_2$) and linear iterations (it) from the Table a) above.}
   \label{fig: scal}
  \end{subfigure}
   \caption{\emph{EMI-GDSW Scalability tests} on $[0,5]$ ms. Condition number ($k_2$) 
 and linear iterations (it) at final time $t=5$ ms.
 Fixed time step $\tau=0.05$. Increasing number of cells from $4$ to $1024$, each discretized with $24\times4$ finite elements. }
  \label{fig: scal_total}
 \end{figure}

\subsubsection{Optimality tests}
We then test the quasi-optimality of our method by fixing the number of cells (subdomains) to be $N=4 \times 4$) and decrease the mesh size $h$, i.e. increase the ratio $H/h$. 
Table \ref{tab: optimality} and Figure \ref{fig: optimality} show again that EMI-GDSW has a much better performace of both AS and unpreconditioned CG.
The polylogarithmic growth of the EMI-GDSW condition number is not easily detectable for this range of $H/h$, but the growth seem definitely less than linear.

 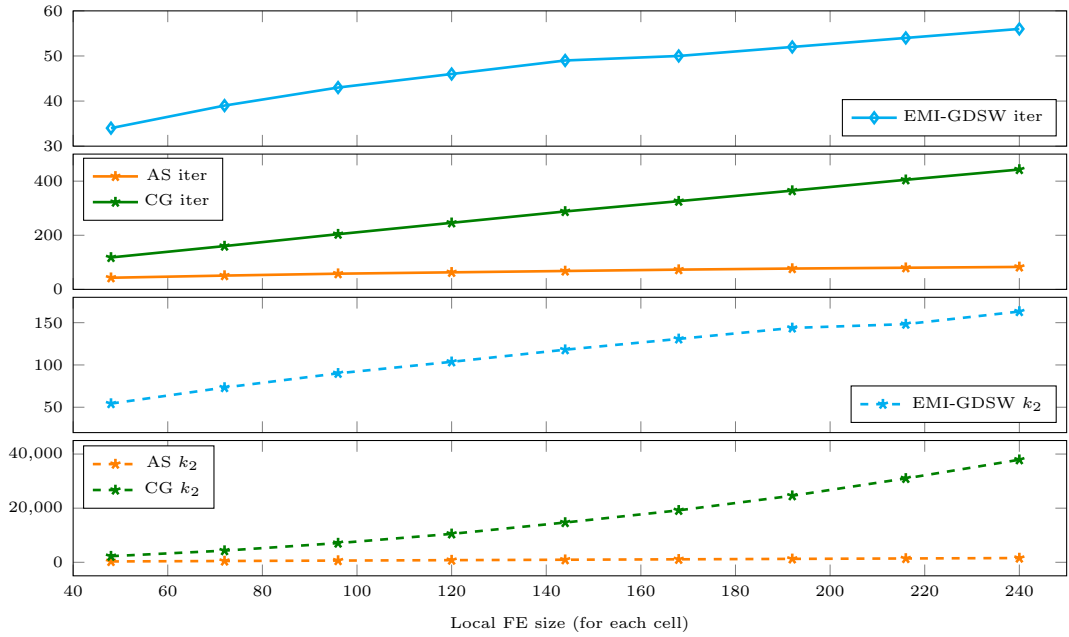
\begin{figure}
	    \centering
\begin{subfigure}{\textwidth}
   \centering
	    \centering
	    \begin{tabular}{cc|rcrcrc}
	        \toprule
	        \multirow{2}{*}{$Lcy$}  &\multirow{2}{*}{$H/h=6 \times Lcy$}  &\multicolumn{2}{c}{EMI-GDSW}    &\multicolumn{2}{c}{AS}     &\multicolumn{2}{c}{CG}    \\
	        &          &$k_2$    &it                 &$k_2$              &it         &$k_2$              &it  \\   
	        \midrule
	        2   &12     &54.3  &34       &313.6  &43        &2.26e+03 &118     \\
	        3   &18     &73.3  &39       &473.7  &51        &4.31e+03 &160     \\
	        4   &24     &90.1  &43       &632.4  &58        &7.09e+03 &204     \\
	        5   &30     &103.7 &46       &791.5  &63        &1.05e+04 &246     \\
	        6   &36     &118.1 &49       &945.1 &68         &1.47e+04 &288     \\
	        7   &42     &130.8 &50       &1.10e+03 &73      &1.92e+04 &326     \\
	        8   &48     &143.8 &52       &1.26e+03 &77      &2.46e+04 &365     \\
	        9   &54     &148.3 &54       &1.41e+03 &80      &3.10e+04 &405     \\
	        10  &60     &163.2 &56       &1.57e+03 &83      &3.79e+04 &443     \\
	        \bottomrule
	    \end{tabular}
	    \caption{Condition number ($k_2$)  and linear iterations (it) at final time $t=5$ ms for EMI-GDSW (left), AS (center), unpreconditioned CG (right).}
	    \label{tab: optimality}
     \vspace{0.7cm}
\end{subfigure}
\begin{subfigure}{\textwidth}
    \centering
    \begin{tikzpicture}
			\begin{groupplot}
				[
				group style={group size=1 by 4, vertical sep=3pt, ylabels at=edge left,
					x descriptions at=edge bottom},
				width=.8\textwidth,
				height=1.8cm,
				scale only axis,
				scaled ticks = false,
				tick label style={/pgf/number format/fixed, /pgf/number format/precision=5},
				tick label style={font=\tiny},
				label style={font=\tiny},
				legend style={font=\tiny},
				xlabel={Local FE size (for each cell)},
				xmin=40, xmax=250,
				unbounded coords=jump]
				
				\nextgroupplot[legend style={at={(0.99,0.21)},anchor=east}, ymin=30, ymax=60]
				\addplot+[cyan, line width=1pt, mark=diamond] table [x=fesize, y=iter_gdsw_opt, col sep=comma, filter discard warning=false, unbounded coords=discard] {plot/results/numer.csv};
				
				\legend{EMI-GDSW iter};
				
				\nextgroupplot[legend style={at={(0.01,0.74)},anchor=west}, ymin=0, ymax=500]
				\addplot+[orange, line width=1pt, mark=star] table [x=fesize, y=iter_as_opt, col sep=comma, filter discard warning=false, unbounded coords=discard] {plot/results/numer.csv};
				
				\addplot+[green!50!black, line width=1pt, mark=star] table [x=fesize, y=iter_cg_opt, col sep=comma, filter discard warning=false, unbounded coords=discard] {plot/results/numer.csv};
				
				\legend{AS iter, CG iter};
				
				\nextgroupplot[legend style={at={(0.99,0.21)},anchor=east}, ymin=20, ymax=180]
				
				\addplot+[cyan, dashed, line width=1pt, mark=star] table [x=fesize, y=cond_gdsw_opt, col sep=comma, filter discard warning=false, unbounded coords=discard] {plot/results/numer.csv};
				
				\legend{EMI-GDSW $k_2$};
				
				\nextgroupplot[legend style={at={(0.01,0.73)},anchor=west}, ymin=-5000, ymax=45000]
				\addplot+[orange, dashed, line width=1pt, mark=star] table [x=fesize, y=cond_as_opt, col sep=comma, filter discard warning=false, unbounded coords=discard] {plot/results/numer.csv};
				
				\addplot+[green!50!black, dashed, line width=1pt, mark=star] table [x=fesize, y=cond_cg_opt, col sep=comma, filter discard warning=false, unbounded coords=discard] {plot/results/numer.csv};
				
				\legend{AS $k_2$, CG $k_2$};
				
			\end{groupplot}
		\end{tikzpicture}
    \caption{Plots of condition number ($k_2$) and linear iterations (it) from the Table a) above.}
    \label{fig: optimality}
    \end{subfigure}
     \caption{\emph{EMI-GDSW optimality tests} on $[0,5]$ ms. Condition number $k_2$  and
    linear iterations (it) at the final time $t=5$ ms. 
    Fixed time step size $\tau=0.05$ and number of $4\times 4$ cells, each discretized with increasing number of finite elements.}
     \label{fig: optimality_total}
    \end{figure}
	
\subsubsection{Dependence on the time step size} \label{sec: dependence tau}
Here we study the convergence rate dependence of the proposed solver on the time step size $\tau$.
We consider $N = 12\times12$ cells, each one discretized with $24\times 4$ Q1 finite elements. 
The time interval is $[0,5]$ ms, where we vary $\tau$ from $0.005$ to $0.1$ ms.
As in our previous study \cite{huynh2023convergence}, we observe that the condition number ($k_2$) and the iteration counts (it) for both unpreconditioned CG and AS solvers increase when the time step size $\tau$ is decreased, see Table \ref{tab: dependence time step} and Fig. \ref{fig: dependence time step}. On the other hand, the EMI-GDSW preconditioner is only marginally affected by the reduction of $\tau$, yielding almost bouded $k_2$ and it. values.
This test shows that the reduction of the time step size $\tau$ does not impair the performance of the GDSW solver, indicating that the bound in Thm. \ref{main_theorem} might not be sharp in $\tau$.

\begin{figure}
	    \centering
\begin{subfigure}{\textwidth}
	    \centering
	    \begin{tabular}{l|cccccc}
	        \toprule
	        \multirow{2}{*}{$\tau$} &\multicolumn{2}{c}{EMI-GDSW}    &\multicolumn{2}{c}{AS}     &\multicolumn{2}{c}{CG}    \\
	                   &$k_2$      &it            &$k_2$      &it         &$k_2$       &it       \\      
	        \midrule
	        $0.005$      &130.7  &57       &9.32e+03   &183       &2.55e+05   &630   \\
	        $0.01$       &127.5  &57       &8.33e+03   &177       &1.55e+05   &627   \\
	        $0.02$       &122.1  &56       &7.25e+03   &168       &9.38e+04   &584   \\
	        $0.05$       &112.2  &54       &5.87e+03   &157       &5.26e+04   &494  \\
	        $0.1$        &132.3  &54       &3.10e+03   &114       &4.26e+04   &468   \\
	        \bottomrule
	    \end{tabular}
	    \caption{Condition number ($k_2$)  and linear iterations (it) at final time $t=5$ ms for EMI-GDSW (left), AS (center), unpreconditioned CG (right).}
 \label{tab: dependence time step}
     \vspace{0.7cm}
\end{subfigure}
\begin{subfigure}{\textwidth}
    \centering
    \begin{tikzpicture}
		\begin{groupplot}
			[
			group style={group size=1 by 4, vertical sep=3pt, ylabels at=edge left,
				x descriptions at=edge bottom},
			width=.8\textwidth,
			height=1.8cm,
			scale only axis,
			scaled ticks = false,
			tick label style={/pgf/number format/fixed, /pgf/number format/precision=5},
			tick label style={font=\tiny},
			xtick={0.005, 0.01, 0.02, 0.05, 0.1},
			label style={font=\tiny},
			legend style={font=\tiny},
			xlabel={Time step $\tau$},
			xmin=0, xmax=0.105, 
			unbounded coords=jump]
			
			\nextgroupplot[legend style={at={(0.99,0.78)},anchor=east}, ymin=53,ymax=58]
			\addplot+[cyan, line width=1pt, mark=diamond] table [x=dt, y=iter_gdsw_dt, col sep=comma, filter discard warning=false, unbounded coords=discard] {plot/results/numer.csv};
			
			\legend{EMI-GDSW iter};
			
			\nextgroupplot[legend style={at={(0.99,0.38)},anchor=east}, ymin=50, ymax=700]
			\addplot+[orange, line width=1pt, mark=star] table [x=dt, y=iter_as_dt, col sep=comma, filter discard warning=false, unbounded coords=discard] {plot/results/numer.csv};
			
			\addplot+[green!50!black, line width=1pt, mark=star] table [x=dt, y=iter_cg_dt, col sep=comma, filter discard warning=false, unbounded coords=discard] {plot/results/numer.csv};
			
			\legend{AS iter, CG iter};
			
			\nextgroupplot[legend style={at={(0.99,0.21)},anchor=east}, ymin=105, ymax=139]
			
			\addplot+[cyan, dashed, line width=1pt, mark=star] table [x=dt, y=cond_gdsw_dt, col sep=comma, filter discard warning=false, unbounded coords=discard] {plot/results/numer.csv};
			
			\legend{EMI-GDSW $k_2$};
			
			\nextgroupplot[legend style={at={(0.99,0.73)},anchor=east}, ymin=-50000, ymax=350000]
			\addplot+[orange, dashed, line width=1pt, mark=star] table [x=dt, y=cond_as_dt, col sep=comma, filter discard warning=false, unbounded coords=discard] {plot/results/numer.csv};
			
			\addplot+[green!50!black, dashed, line width=1pt, mark=star] table [x=dt, y=cond_cg_dt, col sep=comma, filter discard warning=false, unbounded coords=discard] {plot/results/numer.csv};
			
			\legend{AS $k_2$, CG $k_2$};
			
		\end{groupplot}
	\end{tikzpicture}
     \caption{Plots of condition number ($k_2$) and linear iterations (it) from the Table a) above.}
         \label{fig: dependence time step}
 \end{subfigure}
    \caption{\emph{EMI-GDSW dependence on the time step size.} Condition numbers $k_2$ and linear iterations (it)  at the final time $t=5$ ms, when the time step size $\tau$ is increased from $0.005$ to $0.1$. Fixed number of $12\times12$ cells, each discretized with $24\times4$ finite elements.}
    \label{fig: dependence time step_total}
\end{figure}
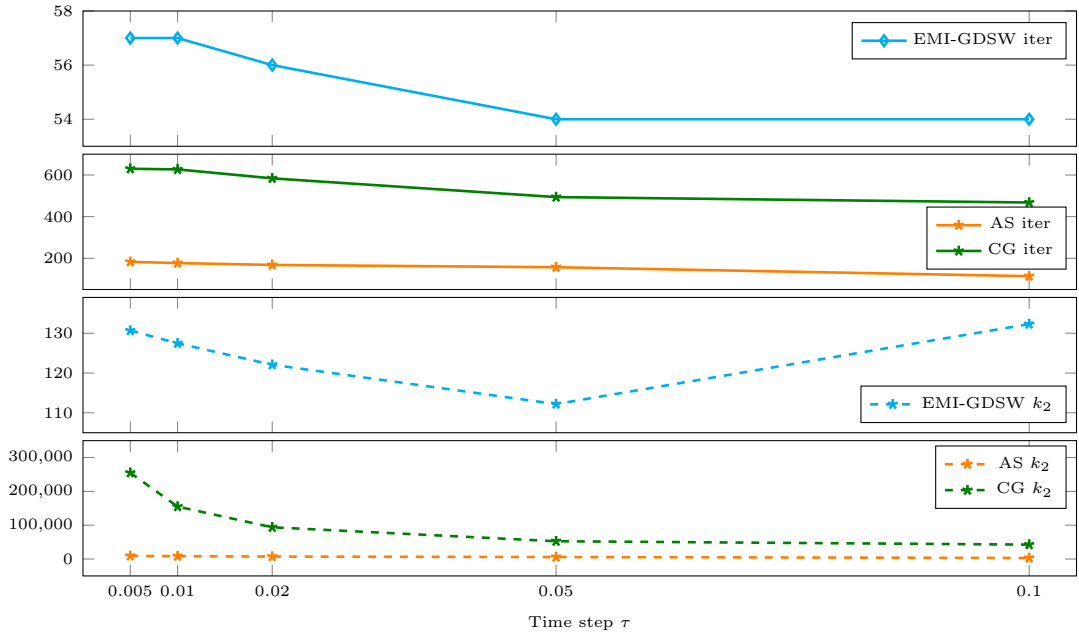

\subsubsection{EMI-GDSW robustness varying the conductivity coefficients}
The last set of tests investigates the robustness of the EMI-GDSW solver with respect to heterogeneous distributions of intracellular diffusion coefficients $\sigma_i$, for $i=1,\dots,N$. 
In particular, we consider three different distributions with $N=8\times 8$ cells, each discretized with $48 \times 8$ Q1 finite elements:\\
- \texttt{checkboard} distribution, where the conductivities alternate between two values $\sigma_i$ and $\sigma_i^{\star}$ in a checkboard fashion;\\ 
- \texttt{capsule} distribution, where we include a inner square block of $4\times 4$ cells with  conductivity $\sigma_i^{\star}$, while in the surrounding cells thee conductivity remain $\sigma_i$;\\
- \texttt{random} distribution, where we randomly generate for each cell the coefficient $\sigma_i^{\star}$. \\
The value $\sigma_i^{\star}$ is obtained by scaling $\sigma_i$ by a factor $10^{-1}$, $10^{-2}$, $10^{-3}$ and $10^{-4}$, except for the \texttt{random} setting, where $\sigma_i^\star = (\text{scaling factor}) \times (\sigma_i + n_\text{rand})$, being $n_\text{rand}$ a randomly generated number between 0 and $10^{-3}$ (in order perturb $\sigma_i$ of a value with the same order of magnitude).
We report in Figures \ref{fig: time evolution 1} and \ref{fig: time evolution 2} three different time snapshots for each distribution, and also including for comparison the homogeneous distribution (\texttt{normal}) with  $\sigma_i = 3 \times 10^{-3}$ in all cells.

Table \ref{tab:robustness} reports both condition number ($k_2$) and number of linear iterations (it) at $t=5$ ms. 
The results show the robustness of EMI-GDSW, since both parameters remain fairly unchanged, with only a slight decrease in the former when reducing $\sigma_i^\star$. 
This is to be expected, since decreasing the diffusion coefficient $\sigma_i^\star$ and considering enough modified subdomains, the mass term from the jump would prevail.

    \begin{figure}
	    \centering
        \texttt{normal} \\
        \vspace{2mm}
	    $t=1$ ms\\ \vspace{0.5mm} \includegraphics[width=0.7\textwidth]{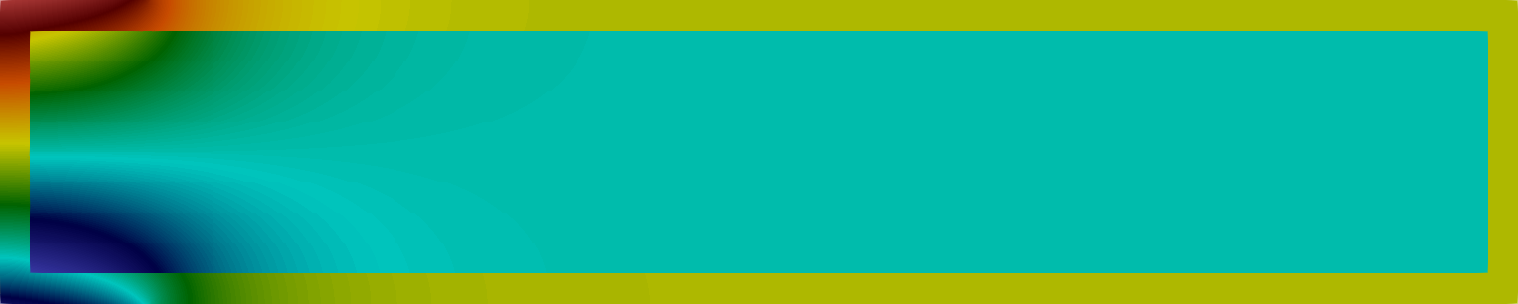}  \\
        \vspace{1mm}
        $t=15$ ms\\ \vspace{0.5mm} \includegraphics[width=0.7\textwidth]
        {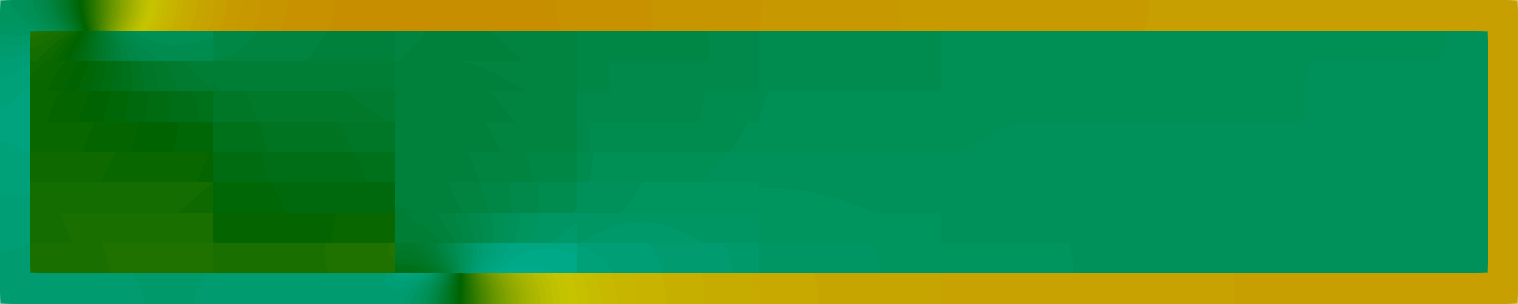} \\
        \vspace{1mm}
        $t=30$ ms\\ \vspace{0.5mm} \includegraphics[width=0.7\textwidth]{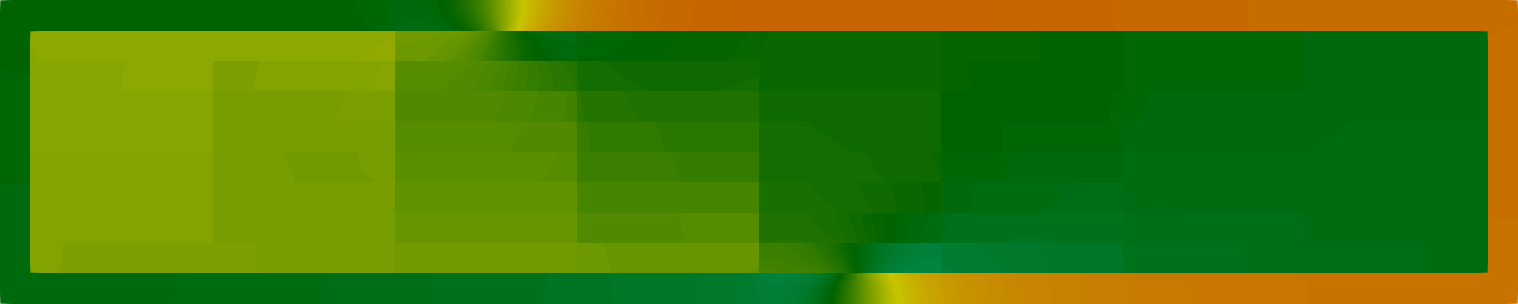} \\
        \vspace{1mm}
        \includegraphics[width=0.6\textwidth]{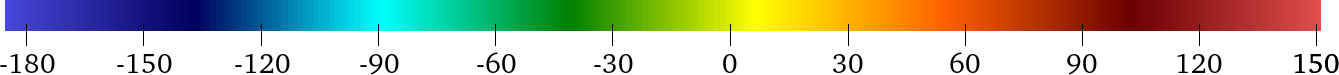} \\
	    \vspace{2mm} \hrule \vspace{2mm}
     
        \texttt{capsule} \\
        \vspace{2mm}
	    $t=1$ ms\\ \vspace{0.5mm} \includegraphics[width=0.7\textwidth]{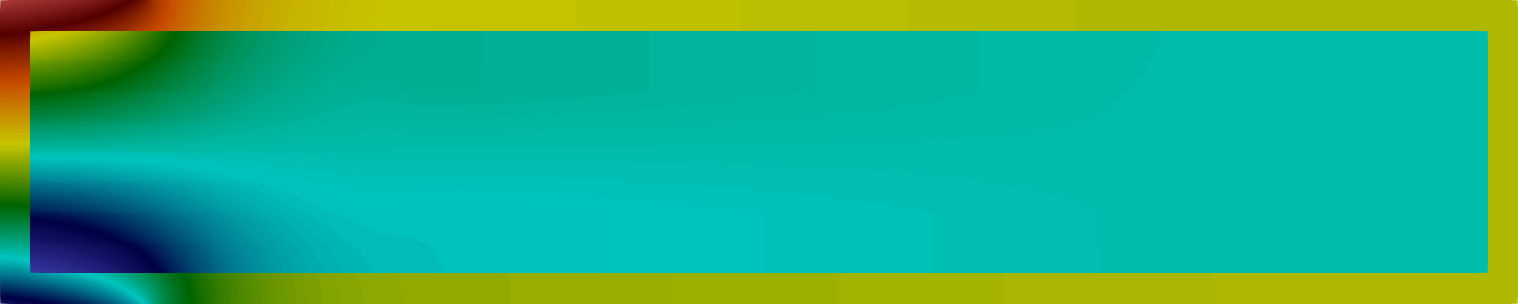} \\
        \vspace{1mm} 
        $t=15$ ms\\ \vspace{0.5mm} \includegraphics[width=0.7\textwidth]{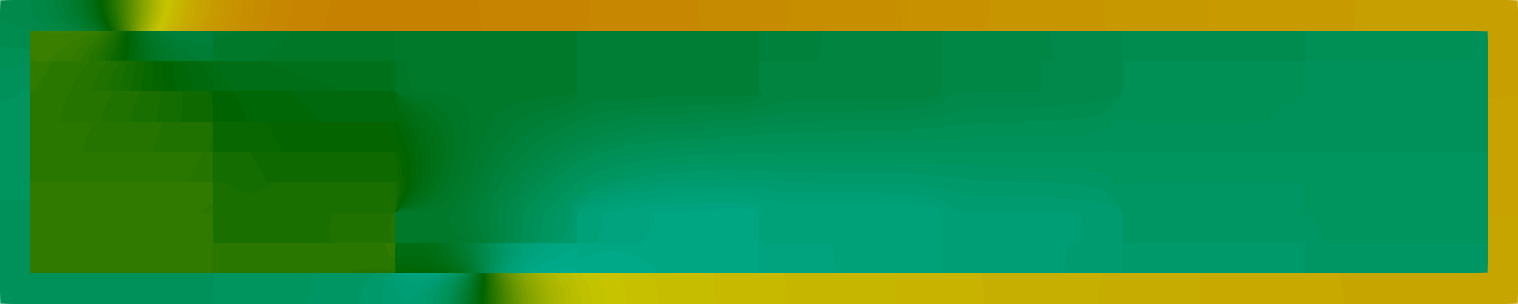} \\
        \vspace{1mm}
        $t=30$ ms\\ \vspace{0.5mm} \includegraphics[width=0.7\textwidth]{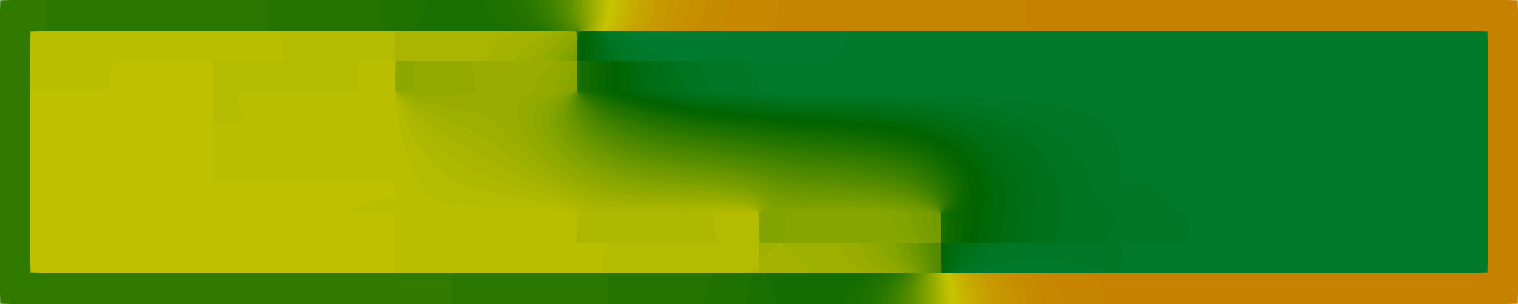} \\
        \vspace{1mm}
        \includegraphics[width=0.6\textwidth]{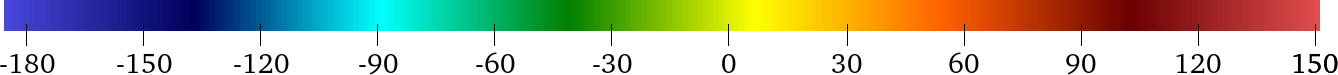}
	    \caption{Snapshots at times $t=1, 15, 30$ ms of the EMI electric potentials for the \texttt{normal} distribution (top)  with $\sigma_i = 3 \times 10^{-3}$ and the \texttt{capsule} distribution (bottom)  with an inner block of $4\times 4$ cells with $\sigma_i^\star=10^{-4} \sigma_i$.}
	    \label{fig: time evolution 1}
	\end{figure}
 
 \begin{figure}
	    \centering
        \texttt{checkboard} \\
        \vspace{2mm}
	    $t=1$ ms\\ \vspace{0.5mm} \includegraphics[width=0.7\textwidth]{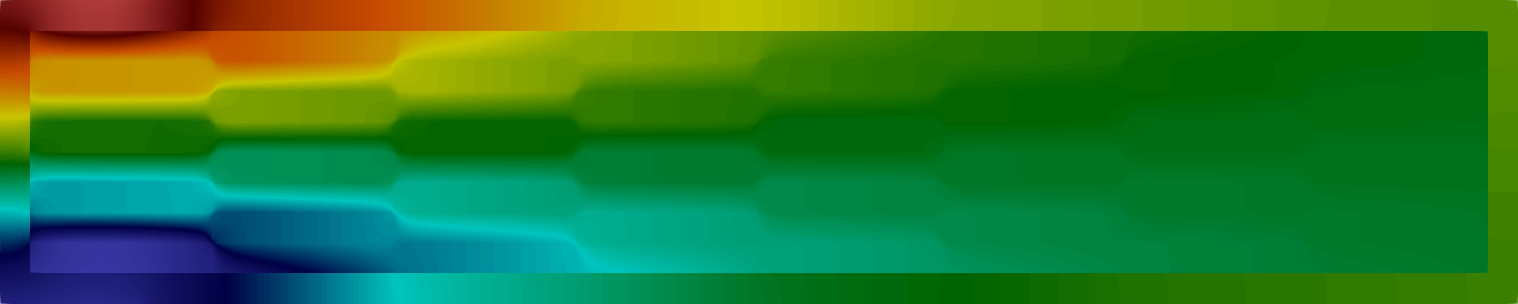}  \\
        \vspace{1mm}
        $t=10$ ms\\ \vspace{0.5mm} \includegraphics[width=0.7\textwidth]{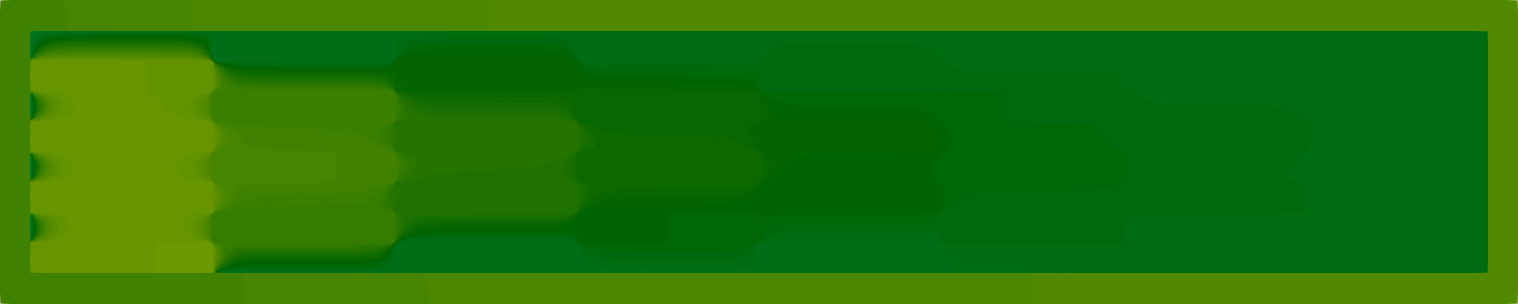} \\
        \vspace{1mm}
        $t=19$ ms\\ \vspace{0.5mm} \includegraphics[width=0.7\textwidth]{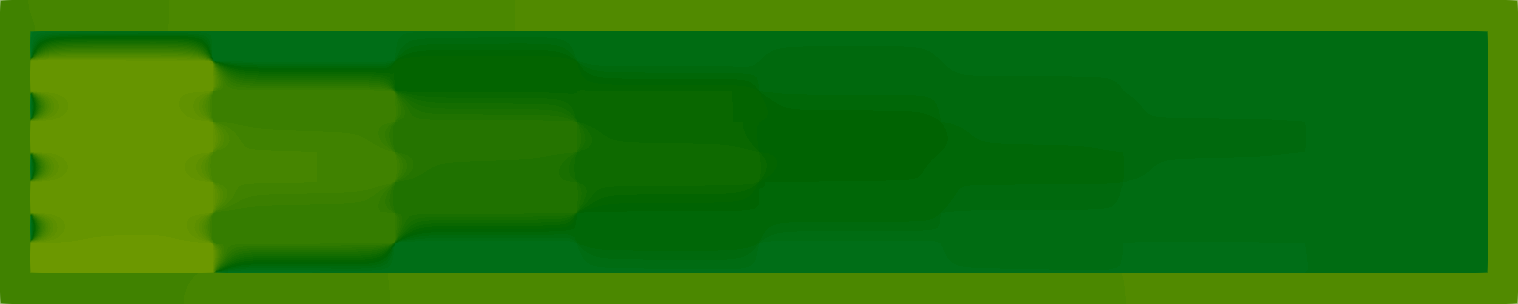} \\
        \vspace{1mm}
        \includegraphics[width=0.6\textwidth]{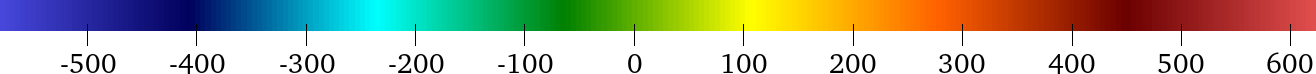} \\
        \vspace{2mm} \hrule \vspace{2mm}
        \texttt{random} \\
        \vspace{2mm}
	    $t=1$ ms\\ \vspace{0.5mm} \includegraphics[width=0.7\textwidth]{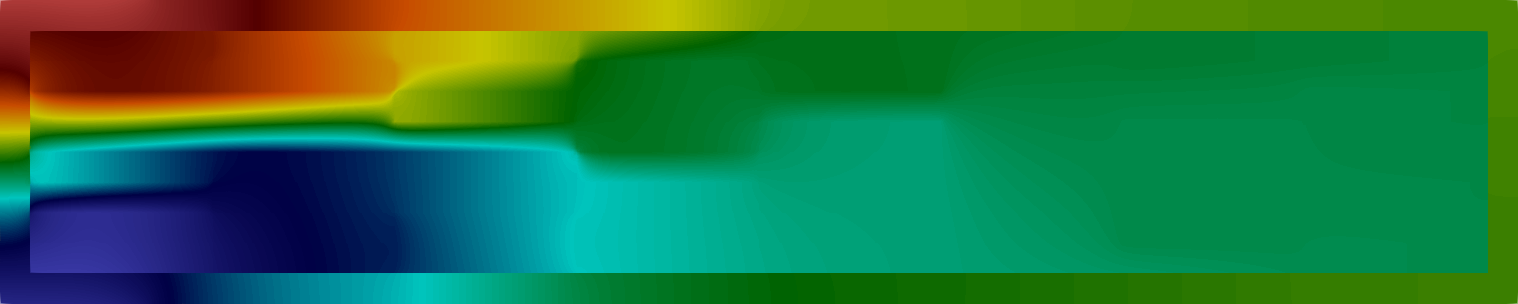}  \\
        \vspace{1mm}
        $t=10$ ms\\ \vspace{0.5mm} \includegraphics[width=0.7\textwidth]{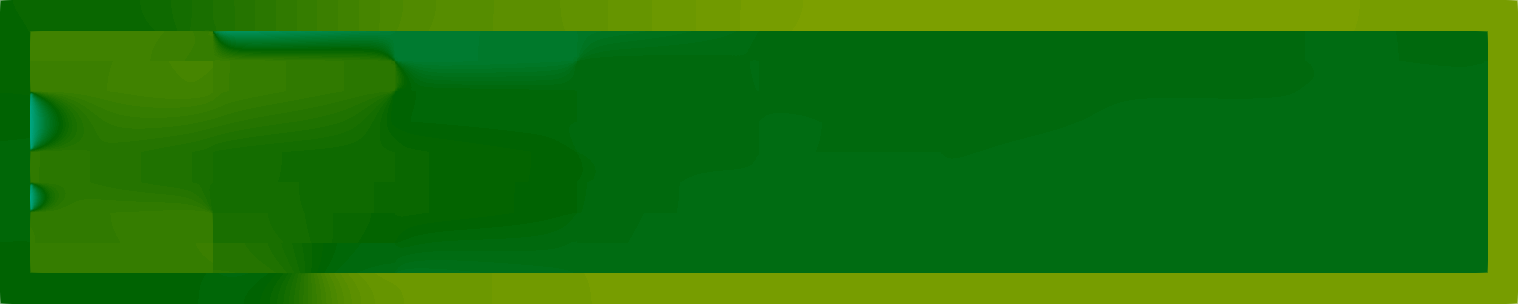} \\
        \vspace{1mm}
        $t=19$ ms\\ \vspace{0.5mm} \includegraphics[width=0.7\textwidth]{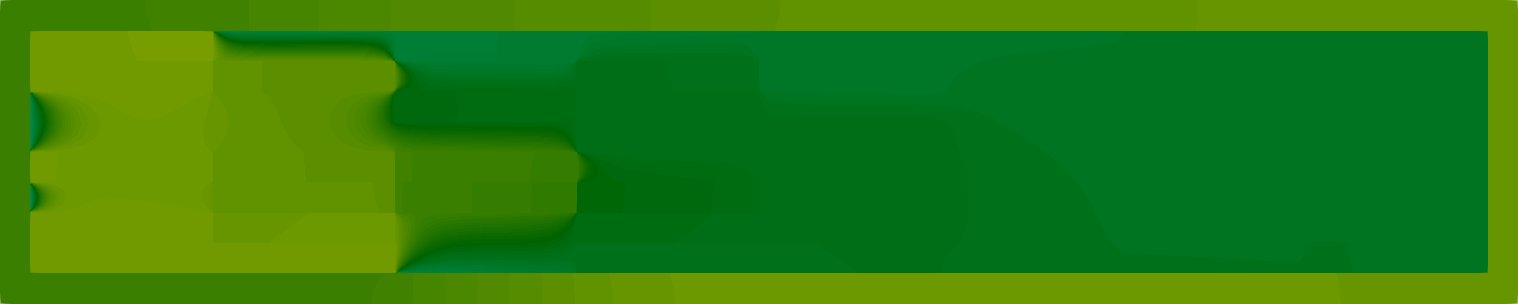} \\
        \vspace{1mm}
        \includegraphics[width=0.6\textwidth]{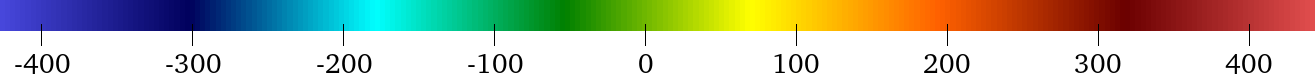} 
	    \caption{Snapshots at times $t=1, 10, 19$ ms of the EMI electric potentials for the \texttt{checkboard} distribution (top) with checkboard-like alternating cells with $\sigma_i^\star=10^{-4} \sigma_i$ and \texttt{random} distribution (bottom) with  $\sigma_i^\star = 10^{-4} \cdot (\sigma_i + n_\text{rand})$, where $n_\text{rand}$ is randomly generated between 0 and $10^{-3}$.}
	    \label{fig: time evolution 2}
	\end{figure}
	
	\begin{table}[!ht]
	    \centering
	    \label{tab: robustness}
	    \begin{tabular}{lr|cccccccccc}
	        \toprule
                 	    & $\alpha$   &\multicolumn{2}{c}{$1$}
                       &\multicolumn{2}{c}{$10^{-1}$}
                       &\multicolumn{2}{c}{$10^{-2}$}    &\multicolumn{2}{c}{$10^{-3}$}     &\multicolumn{2}{c}{$10^{-4}$}    \\
	     &   &$k_2$    &it    &$k_2$    &it          &$k_2$    &it          &$k_2$       &it         &$k_2$         &it  \\   
	        \midrule
\texttt{checkboard} & \multirow{2}*{$\sigma_i^\star = \alpha\sigma_i$} & 167.3 & 62         &137.2  &55       &118.6   &52       &117.3   &53       &117.2   &54   \\
\texttt{capsule}   &      &167.3 & 62     &166.6  &66       &166.7   &65       &167.8   &65       &164.4   &65   \\
         \midrule 	      
 \texttt{random}   & $\sigma_i^\star = \alpha \cdot (\sigma_i + n_\text{rand})$     &  - & -      &191.6  &67       &143.1   &64       &177.8   &68       &137.6   &64   \\
	        \bottomrule
	    \end{tabular}
	    \caption{\emph{EMI-GDSW robustness tests.} EMI-GDSW condition numbers $k_2$ and linear iterations (it) at time $t=5$ ms with three different distributions (\texttt{checkboard}, \texttt{capsule}, \texttt{random}) of the conductivity coefficients $\sigma_i$. Fixed time step $\tau=0.05$ on the interval $[0,5]$ ms. 
     Fixed number of $8\times8$ cells, each discretized with $48\times8$ finite elements. }
    \label{tab:robustness}
 \end{table}

\section{Conclusions} \label{sec: conclusions}
We have designed a generalized Dryja-Smith-Widlund (GDSW) preconditioner for the solution of composite Discontinuous Galerkin discretizations of parabolic reaction-diffusion problems, where the solution can present discontinuities across the domain. These situations naturally arise, for instance, in the context of microscopic biomechanics modeling or in multiscale problems where a model order reduction technique leads to face coupling of different dimensionalities.
We mathematically prove a scalable and quasi-optimal convergence rate bound for the proposed GDSW preconditioner, which is polylogarithmic in the ratio $H\backslash h$ but depends on the time step size, the subdomain diameter, the finite element size and the magnitude of the diffusion coefficient.
Extensive two-dimensional numerical tests confirm this theoretical result for the solution of the cardiac EMI reaction-diffusion model.
Possible future works should address the extension of this result to three-dimensions, as well as investigating the convergence of more sophisticated and advanced GDSW preconditioners, such as adaptive, reduced and multilevel GDSW.


\end{document}